\documentclass[a4paper, 11pt, twoside]{amsart}

\usepackage[top=3cm, bottom=3cm, left=3cm, right=3cm]{geometry}
\usepackage[pdftex]{graphicx}
\usepackage{tikz}
\usetikzlibrary{shapes, arrows, arrows.meta, decorations.markings}

\usepackage{amsmath}
\usepackage{amsthm}
\usepackage{amssymb}
\usepackage{mathtools}
\usepackage[integrals]{wasysym}
\usepackage{stmaryrd}

\usepackage{bbm}
\usepackage[mathcal]{euscript}

\usepackage[utf8]{inputenc}
\usepackage[T1]{fontenc}
\usepackage[english]{babel}
\usepackage{cite}
\usepackage{lmodern}

\usepackage{calc}
\usepackage[inline]{enumitem}
\usepackage[normalem]{ulem}
\usepackage{url}

\usepackage[hidelinks, bookmarks, bookmarksnumbered, pdfstartview={XYZ null null 1.00}]{hyperref}

\newcommand{\theoname}{Theorem}
\newcommand{\lemmname}{Lemma}
\newcommand{\coroname}{Corollary}
\newcommand{\propname}{Proposition}
\newcommand{\definame}{Definition}

\newcommand{\remkname}{Remark}
\newcommand{\explname}{Example}

\theoremstyle{plain}
\newtheorem{theorem}{\theoname}[section]
\newtheorem{lemma}[theorem]{\lemmname}
\newtheorem{corollary}[theorem]{\coroname}
\newtheorem{proposition}[theorem]{\propname}
\newtheorem*{claim}{Claim}
\theoremstyle{definition}
\newtheorem{definition}[theorem]{\definame}
\newtheorem{remark}[theorem]{\remkname}

\newtheorem{notation}[theorem]{Notation}

\DeclareMathOperator{\Lip}{Lip}

\DeclareMathOperator{\diverg}{div}

\DeclareMathOperator{\diff}{d\!}

\DeclarePairedDelimiter{\abs}{\lvert}{\rvert}

\newcommand{\suchthat}{\ifnum\currentgrouptype=16 \mathrel{}\middle|\mathrel{}\else\mid\fi}

\DeclareMathOperator{\MFG}{MFG}
\DeclareMathOperator{\OCP}{OCP}
\DeclareMathOperator{\Adm}{Adm}
\DeclareMathOperator{\Opt}{Opt}

\numberwithin{figure}{section}

\begin{document}

\setlength{\parskip}{1pt plus 1pt minus 1pt}

\newlist{hypotheses}{enumerate}{1}
\setlist[enumerate, 1]{label={\textnormal{(\alph*)}}, ref={(\alph*)}, leftmargin=0pt, itemindent=*}
\setlist[enumerate, 2]{label={\textnormal{(\roman*)}}, ref={(\roman*)}}
\setlist[description, 1]{leftmargin=0pt, itemindent=*}
\setlist[itemize, 1]{label={\textbullet}, leftmargin=0pt, itemindent=*}
\setlist[hypotheses]{label={\textup{(H\arabic*)}}, leftmargin=*}

\title{Multi-population minimal-time mean field games}

\author{Saeed Sadeghi Arjmand}
\address{CMLS, \'Ecole Polytechnique, CNRS, Universit\'e Paris-Saclay, 91128, Palaiseau, France, \& Universit\'e Paris-Saclay, CNRS, CentraleSup\'elec, Inria, Laboratoire des signaux et syst\`emes, 91190, Gif-sur-Yvette, France.}
\email{saeed.sadeghi-arjmand@polytechnique.edu}

\author{Guilherme Mazanti}
\address{Universit\'e Paris-Saclay, CNRS, CentraleSup\'elec, Inria, Laboratoire des signaux et syst\`emes, 91190, Gif-sur-Yvette, France.}
\email{guilherme.mazanti@inria.fr}

\begin{abstract}
In this paper, we consider a mean field game model inspired by crowd motion in which several interacting populations evolving in $\mathbbm R^d$ aim at reaching given target sets in minimal time. The movement of each agent is described by a control system depending on their position, the distribution of other agents in the same population, and the distribution of agents on other populations. Thus, interactions between agents occur through their dynamics. We consider in this paper the existence of Lagrangian equilibria to this mean field game, their asymptotic behavior, and their characterization as solutions of a mean field game system, under few regularity assumptions on agents' dynamics. In particular, the mean field game system is established without relying on semiconcavity properties of the value function.
\end{abstract}


\maketitle
\hypersetup{pdftitle={Multi-population minimal-time mean field games}, pdfauthor={}}

\tableofcontents

\section{Introduction}
Mean field games (MFGs for short) are differential games with a continuum of agents assumed to be rational, indistinguishable, and influenced only by an averaged behavior of other agents through a mean-field type interaction. Following previous works in the economics literature on games with infinitely many agents \cite{Aumann1964Markets, Aumann1974Values, Jovanovic1988Anonymous}, the theory of mean field games has been introduced in 2006 by the simultaneous works of Jean-Michel Lasry and Pierre-Louis Lions \cite{LasryLionen, LasryLionsfr1, LasryLionsfr2}, and of Peter E.~Caines, Minyi Huang, and Roland P.~Malhamé \cite{Huang2003Individual, HuangMalhame, 2HuangMalhame}, motivated by problems in economics and engineering and with the goal of approximating Nash equilibria of games with a large number of symmetric agents. Since their introduction, mean field games have been extensively studied in the literature and several research topics have been addressed, both from theoretical and applied perspectives. The main goal is typically to study equilibria of such games, which are usually characterized as solutions of a system of PDEs, called \emph{MFG system}. We refer to \cite{CardaliaguetNotes, Carmona2018ProbabilisticI, Carmona2018ProbabilisticII, Cardaliaguet2019Master, Gomes2016Regularity} for more details and further references on mean field games.

In this paper, we consider a mean field game model inspired by crowd motion in which a multi-population crowd wishes to arrive at given target sets in minimal time. Motivated by modeling, control, and optimization objectives, the mathematical analysis of crowd motion is the subject of a very large number of works from diverse perspectives \cite{Maury2019Crowds, Gibelli2018Crowd, Cristiani2014Multiscale, Muntean2014Collective, DhelbingIfarkas, LFHenderson, Rosini2013Macroscopic, PiccoliBenedettoTosin, Helbing1995Social}. Among other points of view commonly adopted in the literature, the macroscopic modeling of crowds consists in approximating the location of the finitely many agents in the crowd by a continuous distribution, which is usually assumed to evolve according to some conservation law, and is the natural framework for a mean field game model of crowd motion.

Some previous works on mean field games, such as \cite{Achdou2017YvesBardiMartinoCirantMarco, Achdou2019Lasry, Carlini2018ElisabettaSilva, MeszarosAlparRichard2015SilvaFrancisco, Lachapelle2011Mean, Burger2013Mean, CardaliaguetPierre2, BenamouCarlierSantam, Mazanti2019Minimal, Dweik2020Sharp, DucasseSecond}, have considered mean field games for, or related to, crowd motion. For instance, \cite{Lachapelle2011Mean} proposes a MFG model for a two-population crowd with trajectories perturbed by additive Brownian motion and considers both their stationary distributions and their evolution on a prescribed time interval. Other works also considered multi-population MFGs, such as \cite{Achdou2017YvesBardiMartinoCirantMarco, Carlini2018ElisabettaSilva}, which study in particular a two-population MFG model motivated by urban settlements. The work \cite{Burger2013Mean} considers the fast exit of a crowd, whose agents are perturbed by additive Brownian motion, and proposes a mean field game model, which is studied numerically. Even though \cite{CardaliaguetPierre2} is not originally motivated by the modeling of crowd motion, the MFG model studied in that reference presents a density constraint, preventing high concentration of agents, which is a natural assumption in some crowd motion models. We refer to \cite{MeszarosAlparRichard2015SilvaFrancisco} for second-order mean field games with density constraints. Numerical simulations for some variational mean field games related to crowd motion are presented in \cite{BenamouCarlierSantam}.

The present work is more closely related to \cite{Mazanti2019Minimal, Dweik2020Sharp, DucasseSecond}, which present some particular characteristics with respect to most of the MFG literature. Firstly, contrarily to a large part of the MFG literature but similarly to \cite{Burger2013Mean} and some other works with motivation unrelated to crowd motion, such as \cite{Graber2020Mean}, references \cite{Mazanti2019Minimal, Dweik2020Sharp, DucasseSecond} consider mean field games in which agents do not necessarily stop all at the same time, but may instead have different stopping times, which are actually the main part of the optimization criterion. Secondly, most of MFG models consider that agents are free to choose their speed, with high speeds penalized in the optimization criterion of each agent, but \cite{Mazanti2019Minimal, Dweik2020Sharp, DucasseSecond} assume instead that agents may move only up to a certain maximal speed, which depends on the average distribution of agents around their position. As detailed in \cite{Mazanti2019Minimal}, this assumption is intended to model crowd motion situations in which an agent may not be able to move faster by simply paying some additional cost, since the congestion provoked by other agents may work as a physical barrier for the agent to increase their speed. We refer to \cite{Mazanti2019Minimal, Dweik2020Sharp, DucasseSecond} for more details on the motivation of the model and its relation to other crowd motion models.

Similarly to \cite{Mazanti2019Minimal, Dweik2020Sharp, DucasseSecond}, the MFG studied in this work assumes that agents want to minimize their time to reach a certain target set, their optimal control problem being thus with a free final time, and that their maximal speed is bounded in terms of the density of agents around their position. Several novelties are considered in the MFG from the present paper. Firstly, we assume that the agents taking part in the game are not all identical, but are instead subdivided in $N$ populations. Each population $i \in \{1, \dotsc, N\}$ may present different dynamics and different target sets. This additional assumption brings no major difficulty in the analysis of the MFG but allows for the representation of more realistic situations, such as two populations in a corridor starting at opposite sides, each one wanting to reach the other side in minimal time. We also allow for the interaction of an agent with other agents of the same population to be different than their interaction with agents of other populations, in order to model the fact that it may be easier to move with other agents that want to reach the same target, and hence move in the same general direction, than to move in a crowd of people going on different directions.

Another novelty from the present paper with respect to \cite{Mazanti2019Minimal, Dweik2020Sharp, DucasseSecond} is to consider that agents move on $\mathbbm R^d$, instead of on a compact subset of $\mathbbm R^d$. Lack of compactness of the state space brings additional difficulties in the analysis of the MFG, in particular since we are interested in situations in which the initial distribution of agents is not necessarily compactly supported, but these difficulties can be overcome by exploiting suitable properties of optimal trajectories. In particular, the time for an agent to reach their target set is no longer uniformly bounded, but we are able to provide sharp bounds on the convergence rate of the distribution of agents towards their limit distribution concentrated in the target set. We also remark that, contrarily to \cite{Mazanti2019Minimal, Dweik2020Sharp, DucasseSecond}, the target sets are not assumed to be the boundary of a compact domain, but can be arbitrary nonempty closed subsets of $\mathbbm R^d$.

Finally, we also relax the regularity assumptions on the dynamics of agents from \cite{Mazanti2019Minimal, Dweik2020Sharp}, requiring only continuity with respect to the distributions of other agents and Lipschitz continuity with respect to the space variable. In those references, similar assumptions were used to prove existence of Lagrangian equilibria, but additional regularity assumptions were required to characterize such equilibria as solutions of a MFG system. These additional assumptions were used in \cite{Mazanti2019Minimal, Dweik2020Sharp} to obtain semiconcavity of the value function of the optimal control problem solved by each agent, which is a key step to obtain differentiability of the value function along optimal trajectories and hence deduce that the velocity field in the continuity equation of the MFG system is well-defined and continuous on the support of the distribution of agents. By not requiring these additional regularity assumptions, the present paper uses instead different techniques to study the velocity field appearing in the continuity equation, based on a detailed study of some properties of optimal trajectories, which allows us to obtain the MFG system without relying on the semiconcavity of the value function. This is probably one of the main contributions of the present paper and brings several interesting perspectives, in particular since these techniques might be adapted to other MFG models in which semiconcavity of the value function is known not to hold, such as in some MFGs with state constraints. We also refer the interested reader to \cite{CannarsaPiermarco, Cannarsa2019C11, Cannarsa2021Mean} for other approaches for dealing with MFGs with state constraints.

The notion of MFG equilibrium is formulated in this paper in a Lagrangian setting, which describes the motion of agents by a measure on the set of all possible trajectories, instead of the more classical approach consisting in describing the evolution of agents through a time-dependent measure on the space state. The Lagrangian approach is classical in optimal transport problems (see, e.g., \cite{Ambrosio1, FilippoSantam, CedricVillani, YannBrenier, CarlierJimenez, BernotMarcCaselles}) and has also recently been used in several works on mean field games \cite{BenamouCarlierSantam, CannarsaPiermarco, CardaliaguetPierre, CardaliaguetPierre2, Mazanti2019Minimal, Dweik2020Sharp}.

This paper is organized as follows. Section~\ref{SecNotation} settles the main notations used in the paper, while Section~\ref{sec MFG model} describes the mean field game model considered here together with its associated optimal control problem solved by each agent, and presents the main tools used in the sequel. Section~\ref{sec OCP} presents the important results on the optimal control problem needed for the sequel of the paper. The main results on our MFG model are provided in Section~\ref{sec MFGs}, which proves the existence of an equilibrium, studies its asymptotic behavior at large times, and shows that the distribution of the agents and the value function of the optimal control problem solved by each agent can be characterized by the system of partial differential equations known as MFG system.

\section{Notation and preliminary definitions}
\label{SecNotation}

In this paper, $N$ and $d$ are fixed positive integers. The set of nonnegative real numbers is denoted by $\mathbbm R_+$. We denote the usual Euclidean norm in $\mathbbm R^d$ by $\abs{\cdot}$ and the unit sphere in $\mathbbm R^d$ by $\mathbbm S^{d-1}$. Given $x \in \mathbbm R^d$ and $R \geq 0$, we write $B(x, R)$ for the closed ball centered at $x$ and of radius $R$. When $x = 0$, this ball is denoted simply by $B_R$. We use $\mathcal P(\mathbbm R^d)$ to denote the set of all Borel probability measures on $\mathbbm R^d$, which is assumed to be endowed with the topology of weak convergence of measures.

Given two sets $A, B$, a set-valued map $F: A \rightrightarrows B$ is a map that, to each $a \in A$, associates a (possibly empty) set $F(a) \subset B$.

Recall that, for two metric spaces $X$ and $Y$ endowed with their Borel $\sigma$-algebras and a Borel map $f:X \to Y$, the pushforward of a measure $\mu$ on $X$ through $f$ is the measure $f_{\#} \mu$ on $Y$ defined by
$$
f_{\#} \mu (B) = \mu(f^{-1}(B))
$$
for every Borel subset $B$ of $Y$. We extend the pushforward notation componentwise to vectors of measures: if $\pmb\mu = (\mu_1, \dotsc, \mu_N)$ with $\mu_i$ a measure on $X$ for every $i \in \{1, \dotsc, N\}$, then we set $f_{\#}\pmb\mu = (f_{\#}\mu_1, \dotsc, f_{\#}\mu_N)$.

We define, for $p \in [1, +\infty)$, the set
$$
\mathcal{P}_p(\mathbbm R^d)=\left\{ \mu \in \mathcal{P}(\mathbbm R^d) \suchthat \int_{\mathbbm R^d} \abs*{x}^p \diff \mu(x) < +\infty \right\}.
$$
We endow $\mathcal{P}_p(\mathbbm R^d)$ with the usual Wasserstein distance $\mathbf{W}_p$, defined by
\begin{equation}
\label{eq:defi-Wasserstein}
\mathbf{W}_p(\mu,\nu)=\inf \left\{ \int_{\mathbbm R^d \times \mathbbm R^d} \abs{x - y}^p \diff\lambda(x,y) \suchthat \lambda \in \Pi(\mu,\nu) \right\} ^{1\slash p} ,
\end{equation}
where $\Pi(\mu,\nu)= \left\{ \lambda \in \mathcal{P}(\mathbbm R^d \times \mathbbm R^d) \suchthat {\pi_{1}}_{\#} \lambda=\mu,\, {\pi_{2}}_{\#} \lambda=\nu \right\}$ and $\pi_1,\, \pi_2 : \mathbbm R^d \times \mathbbm R^d \to \mathbbm R^d$ denote the canonical projections onto the first and second factors of the product $\mathbbm R^d \times \mathbbm R^d$, respectively.

Given two metric spaces $X$ and $Y$ and $M > 0$, $\mathbf C(X; Y)$, $\Lip(X; Y)$, and $\Lip_M(X; Y)$ denote, respectively, the set of all continuous functions from $X$ to $Y$, the set of all Lipschitz continuous functions from $X$ to $Y$, and the subset of $\Lip(X; Y)$ containing only those functions whose Lipschitz constant is at most $M$. %When $X = \mathbbm R_+$, we omit it from the previous notations and write simply $\mathbf C(Y)$, $\Lip(Y)$, and $\Lip_M(Y)$, respectively.

For $t\in \mathbbm R_+$, we denote by $e_t:\mathbf{C}(\mathbbm R_+; \mathbbm R^d) \to \mathbbm R^d$ the evaluation map at time $t$, defined by $e_t(\gamma)=\gamma(t)$ for every $\gamma \in \mathbf C(\mathbbm R_+; \mathbbm R^d)$. We remark that $\mathbf C(\mathbbm R_+; \mathbbm R^d)$, endowed with the topology of uniform convergence on compact sets, is a Polish space, which is complete when endowed, for instance, with the metric $\mathbf d$ given by
\begin{equation}
\label{eq:defi-d-cont}
\mathbf d(\gamma_1,\gamma_2) = \sum_{n>0} \frac{1}{2^n}\frac{\sup_{t \in [0,n]}\abs*{\gamma_1(t)-\gamma_2(t)}}{1+\sup_{t \in [0,n]}\abs*{\gamma_1(t)-\gamma_2(t)}}
\end{equation}
for $\gamma_1, \gamma_2 \in \mathbf C(\mathbbm R_+; \mathbbm R^d)$. Whenever needed, we assume in the sequel that $\mathbf C(\mathbbm R_+; \mathbbm R^d)$ is endowed with this metric.

\section{The MFG model} \label{sec MFG model}

For $i \in \{1, \dotsc, N\}$, let $\Gamma_i \subset \mathbbm R^d$ be a closed nonempty set, $K_i:\mathcal{P}(\mathbbm R^d) \times \mathcal{P}(\mathbbm R^d)^{N-1} \times \mathbbm R^d \to \mathbbm R_+$, $m_0^i \in \mathcal P(\mathbbm R^d)$, and denote for simplicity $\mathbf\Gamma = (\Gamma_1, \dotsc, \Gamma_N)$, $\mathbf K = (K_1, \dotsc, K_N)$, and $\mathbf{m_0} = (m_0^1, \dotsc, m_0^N)$. We consider in this paper the following mean field game, denoted by $\MFG(\mathbf\Gamma, \mathbf K, \mathbf{m_0})$: $N$ populations evolve in the space $\mathbbm R^d$ and, for $i \in \{1, \dotsc, N\}$, the distribution of the $i$-th population at time $t \geq 0$ is described by a probability measure $m_{t}^{i} \in \mathcal{P}(\mathbbm R^d)$. The aim of each agent of population $i$ is to minimize their time to reach their \emph{target set} $\Gamma_i$ and, in order to model congestion, we assume that the speed of an agent of population $i$ at a position $x$ in time $t$ is bounded by $K_i(m_{t}^{i}, \hat{m}_{t}^{i}, x)$, where $\hat{m}_{t}^{i} \in \mathcal P(\mathbbm R^d)^{N-1}$ describes the distribution of agents in the other populations and is defined by
\begin{equation}
\label{eq:hat_m_i}
\hat m_t^i = (m_t^1, \dotsc, m_t^{i-1}, m_t^{i+1}, \dotsc , m_t^{N}).
\end{equation}
More precisely, we assume that the movement of a representative agent of population $i$ is described by the control system
\begin{equation}
\label{eq:MFG-dynamics}
\dot\gamma(t) = K_i(m_t^i, \hat m_t^i, \gamma(t)) u(t), \qquad u(t) \in B_1,
\end{equation}
where $\gamma(t) \in \mathbbm R^d$ is the state of the agent and $u(t)$ is their control at time $t$, the control being constrained to remain in the closed unit ball $B_1$.

In order to properly model congestion through the functions $K_1, \dotsc , K_N$, a reasonable assumption is that $K_i(\mu_i, \hat \mu_i, x)$ is small when the measures $\mu_1, \dotsc , \mu_N$ are large around $x$, and that larger values of $\mu_j$, $j\neq i$, are more penalized than larger values of $\mu_i$, to model the fact that an agent moving with their own population is less penalized than if this same agent moves in the middle of another population going potentially in another direction. A possible form for each $K_i$ is 
\[
K_i(\mu_i, \hat \mu_i, x)= g\left(\int_{\mathbbm R^d}\chi(x-y)\diff \mu_i(y)+ \sum_{\substack{j=1\\ j\neq i}}^N \lambda_j \int_{\mathbbm R^d}\chi(x-y)\diff \mu_j(y)\right),
\]
where $g:\mathbbm R_+ \to \mathbbm R_+^\ast$ is decreasing, $\chi:\mathbbm R^d \to \mathbbm R_+$ is a smooth convolution kernel, and $\lambda_j > 1$ is a constant for $j\in \{1, \dotsc , N\}\setminus \{i\}$. 
Let us point out that we do not assume this specific form of $K_i$ in the sequel but, under suitable regularity assumptions on $g$ and $\chi$, such a $K_i$ satisfies assumptions \ref{HypoMFG-K-Bound} and \ref{HypoMFG-K-Lip} stated below as well as assumption \ref{HypoMFG-K-loc-Lip-Wass} from Section~\ref{SecMFGSyst} (see, e.g., \cite[Proposition~3.1]{Mazanti2019Minimal} for a similar result).

The trajectory $\gamma$ of an agent in population $i$ depends on the distribution of agents of population $i$ and also on that of agents of other populations, since the speed of $\gamma$ should not exceed $K_i(m_{t}^{i}, \hat{m}_{t}^{i}, \gamma(t))$. On the other hand, the distributions $m_t^i$ depend on how agents choose their trajectories. We are interested here in \emph{equilibrium} situations, i.e., situations in which, starting from time evolutions of the distributions of agents $m^i:\mathbbm R_+ \to \mathcal{P}(\mathbbm R^d)$, the trajectories chosen by agents induce evolutions of the initial distribution of agents $m_0^i$ that are precisely given by $m_t^i$.

To provide a more precise description of $\MFG(\mathbf\Gamma, \mathbf K, \mathbf{m_0})$, we now introduce an auxiliary optimal control problem. Given $\Gamma \subset \mathbbm R^d$ nonempty and closed and $k: \mathbbm R_+ \times \mathbbm R^d \to \mathbbm R_+$, we consider the optimal control problem $\OCP(\Gamma, k)$ in which an agent evolving in $\mathbbm R^d$ wants to reach $\Gamma$ in minimal time, their speed at position $x$ and time $t$ being bounded by $k(t, x)$. For this optimal control problem, $k$ does not depend on the density of the agents and is considered as a given function. The relation between the optimal control problem $\OCP(\Gamma, k)$ and the mean field game $\MFG(\mathbf\Gamma, \mathbf K, \mathbf{m_0})$ is that, for every population $i \in \{1, \dotsc, N\}$, an agent of population $i$ solves $\OCP(\Gamma_i, k_i)$, where $k_i$ is defined by $k_i(t, x) = K_i(m_t^i, \hat m_t^i, x)$ for $t \geq 0$ and $x \in \mathbbm R^d$.

\begin{definition}
Let $\Gamma \subset \mathbbm R^d$ be nonempty and closed and $k: \mathbbm R_+ \times \mathbbm R^d \to \mathbbm R_+$.
\begin{enumerate}[label={(\alph*)}]
    \item A curve $\gamma \in \Lip(\mathbbm R_+; \mathbbm R^d)$ is said to be \emph{admissible} for $\OCP(\Gamma, k)$ if it satisfies $\abs{\Dot{\gamma}(t)} \le k(t, \gamma(t))$ for almost every $t \in \mathbbm R_+$. The set of all admissible curves is denoted by $\Adm(k)$.  

    \item Let $t_0 \in \mathbbm R_+$. The \emph{first exit time} after $t_0$ of a curve $\gamma \in \Lip(\mathbbm R_+; \mathbbm R^d)$ is the number $\tau_\Gamma(t_0, \gamma) \in [0, +\infty]$ defined by
    $$
    \tau_\Gamma(t_0, \gamma) = \inf \{t \ge 0 \suchthat \gamma(t + t_0) \in \Gamma \}.
    $$
				
    \item Let $t_0 \in \mathbbm R_+$ and $x_0 \in \mathbbm R^d$. A curve $\gamma \in \Lip(\mathbbm R_+; \mathbbm R^d)$ is said to be an \emph{optimal trajectory} for $(\Gamma, k, t_0, x_0)$ if $\gamma \in \Adm(k)$, $\gamma(t) = x_0$ for every $t \in [0, t_0]$, $\tau_\Gamma(t_0, \gamma) < +\infty$, $\gamma(t) = \gamma(t_0 + \tau_\Gamma(t_0, \gamma)) \in \Gamma$ for every $t \in [t_0 + \tau_\Gamma(t_0, \gamma), +\infty)$, and
    \begin{equation}\label{min exit time}
        \tau_\Gamma(t_0, \gamma) = \inf_{\substack{\beta \in Adm(k) \\ \beta(t_0) = x_0}} \tau_\Gamma(t_0, \beta).
    \end{equation}
		The set of all optimal trajectories for $(\Gamma, k, t_0, x_0)$ is denoted by $\Opt(\Gamma, k, t_0, x_0)$. 
\end{enumerate}
\end{definition}

Note that admissible curves $\gamma$ for $\OCP(\Gamma, k)$ are trajectories of the control system
\begin{equation}\label{optimal-control}
\dot\gamma(t) = k(t, \gamma(t)) u(t),
\end{equation}
where the measurable function $u:\mathbbm R_+ \to B_1$ is the control associated with $\gamma$. The control system \eqref{optimal-control} is \emph{nonautonomous}, since $k$ explicitly depends on $t$.

We now provide the definition of Lagrangian equilibrium (which we refer to simply as \emph{equilibrium} in this paper for simplicity) of $\MFG(\mathbf\Gamma, \mathbf K, \mathbf{m_0})$.

\begin{definition}\label{Equilibrium}
Let $\mathbf{m_0} = (m_0^1, \dotsc, m_0^N) \in \mathcal P(\mathbbm R^d)^N$, $\mathbf \Gamma = (\Gamma_1, \dotsc, \Gamma_N)$, and $\mathbf K = (K_1, \dotsc, K_N)$ with $\Gamma_i \subset \mathbbm R^d$ nonempty and closed and $K_i: \mathcal P(\mathbbm R^d) \times \mathcal P(\mathbbm R^d)^{N-1} \times \mathbbm R^d \to \mathbbm R_+$ for every $i \in \{1, \dotsc, N\}$. A vector of measures $\mathbf Q = (Q_1, \dotsc, Q_N) \in \mathcal{P}(\mathbf{C}(\mathbbm R_+; \mathbbm R^d))^N$ is called a \emph{(Lagrangian) equilibrium} for $\MFG(\mathbf\Gamma, \mathbf K, \mathbf{m_0})$ if ${e_0}_{\#} \mathbf Q = \mathbf m_0$ and, for every $i \in \{1, \dotsc, N\}$, $Q_i$-almost every $\gamma$ is optimal for $(\Gamma_i, k_i, 0, \gamma(0))$, where $k_i:\mathbbm R_+ \times \mathbbm R^d \to \mathbbm R_+$ is defined for $(t, x) \in \mathbbm R_+ \times \mathbbm R^d$ by $k_i(t, x) = K_i(m_t^{i}, \hat{m}_{t}^{i}, x)$, $m_t^i = {e_t}_{\#} Q_i$, and $\hat{m}_{t}^{i}$ is given by \eqref{eq:hat_m_i}.
\end{definition}

Let us now state the base assumptions on the data of $\MFG(\mathbf\Gamma, \mathbf K, \mathbf{m_0})$ and $\OCP(\Gamma, k)$ used throughout this paper. Concerning $\MFG(\mathbf\Gamma, \mathbf K, \mathbf{m_0})$, we shall always assume the following hypotheses to be satisfied.

\begin{hypotheses}
\item\label{HypoMFG-Gamma} For $i \in \{1, \dotsc, N\}$, $\Gamma_i$ is a nonempty closed subset of $\mathbbm R^d$.
\item\label{HypoMFG-K-Bound} There exist positive constants $K_{\min},\, K_{\max}$ such that, for every $i \in \{1, \dotsc, N\}$, $K_i: \mathcal P(\mathbbm R^d) \times \mathcal P(\mathbbm R^d)^{N-1} \times \mathbbm R^d \to \mathbbm R_+$ is continuous and $K_i(\mu, \nu, x) \in [K_{\min}, K_{\max}]$ for every $(\mu, \nu, x) \in \mathcal P(\mathbbm R^d) \times \mathcal P(\mathbbm R^d)^{N-1} \times \mathbbm R^d$.
\item\label{HypoMFG-K-Lip} The functions $K_i$ are Lipschitz continuous with respect to their third variable, uniformly with respect to the first two variables, i.e., there exists $L > 0$ such that, for every $i \in \{1, \dotsc, N\}$, $\mu \in \mathcal P(\mathbbm R^d)$, $\nu \in \mathcal P(\mathbbm R^d)^{N-1}$, and $x_1, x_2 \in \mathbbm R^d$, we have
\[
\abs{K_i(\mu, \nu, x_1) - K_i(\mu, \nu, x_2)} \leq L \abs{x_1 - x_2}.
\]
\end{hypotheses}

As for $\OCP(\Gamma, k)$, we always assume the following hypotheses to be satisfied.

\begin{hypotheses}[resume]
\item\label{HypoOCP-Gamma} The set $\Gamma$ is a nonempty closed subset of $\mathbbm R^d$.
\item\label{HypoOCP-k-Bound} There exist positive constants $K_{\min},\, K_{\max}$ such that $k: \mathbbm R_+ \times \mathbbm R^d \to \mathbbm R_+$ is continuous and $k(t, x) \in [K_{\min}, K_{\max}]$ for every $(t, x) \in \mathbbm R_+ \times \mathbbm R^d$.
\item\label{HypoOCP-k-Lip} The function $k$ is locally Lipschitz continuous with respect to its second variable, uniformly with respect to the first variable, i.e., for every $R > 0$, there exists $L > 0$ such that, for every $t \in \mathbbm R_+$ and $x_1, x_2 \in B_R$, we have
\[
\abs{k(t, x_1) - k(t, x_2)} \leq L \abs{x_1 - x_2}.
\]
\end{hypotheses}

In the sequel of the paper, we always use the following notation.  
\begin{notation}\label{Notation-Phi}
Given $m_0^1, \dotsc, m_0^N \in \mathcal P(\mathbbm R^d)$, we denote by $\phi: \mathbbm R_+ \to \mathbbm R_+$ the function defined for $R\ge 0$ by $\phi(R) = \min_{i \in \{1, \dotsc, N\}} m_0^i(B_R)$.
\end{notation}
Notice that $\phi$ is nondecreasing and satisfies $\lim_{R \to +\infty} \phi(R) = 1$ and $m_0^i(B_R) \geq \phi(R)$ for every $i \in \{1, \dotsc, N\}$ and $R \geq 0$.

\begin{remark}
Even though this paper considers multi-population mean field games, our techniques also apply to single-population mean field games, in which the function $K_i$ in \eqref{eq:MFG-dynamics} is replaced by a function $K$ depending on the distribution $m_t$ of the single population at time $t$ and on the position $\gamma(t)$ of an agent. We chose to consider the multi-population setting due to the fact that it is closer to applications, since, in most crowd motion situations in practice, different parts of crowd may wish to reach different target sets, such as people taking different exists in a metro station. Moreover, several works such as \cite{Mazanti2019Minimal, Dweik2020Sharp} already consider single-population minimal-time mean field games, although with more restrictive assumption than here, and there is not much additional difficulty when considering directly the multi-population case.
\end{remark}

\section{Preliminary results on the optimal control problem}\label{sec OCP}

In this section, we collect the main properties of the optimal control problem $\OCP(\Gamma, k)$ that will be of use in the sequel of the paper. Note that $\OCP(\Gamma, k)$ is a minimal-time optimal control problem, which is a classic subject in the optimal control literature (see, e.g., \cite{BardiDolcetta, CannarsaPiermarcoSinestrari, Clarke, PontryagainBolttyanskiiGam}), but the assumptions \ref{HypoOCP-Gamma}--\ref{HypoOCP-k-Lip} on $\OCP(\Gamma, k)$ allow for less smooth $\Gamma$ and $k$ than those typically considered in the literature. Minimal-time optimal control problems have also been studied in connection with mean field games, for instance in \cite{Mazanti2019Minimal, Dweik2020Sharp}, the main difference with respect to the present paper being that those references consider optimal control problems in a compact state space, whereas the state space in the present paper is $\mathbbm R^d$.

The first property of $\OCP(\Gamma, k)$ that we consider is the existence of optimal trajectories, stated in the proposition below. Its proof can be carried out by standard techniques based on minimizing sequences and using the relative compactness of bounded subsets of $\Lip_{K_{\max}}(\mathbbm R_+; \mathbbm R^d)$ in the topology of $\mathbf C(\mathbbm R_+; \mathbbm R^d)$ and is omitted here for simplicity (see, e.g., \cite[Theorem~8.1.4]{CannarsaPiermarcoSinestrari} for a similar proof in the case of a more general optimal exit time problem for an autonomous control system).

\begin{proposition}
\label{PropExistOpt}
Consider the optimal control problem $\OCP(\Gamma, k)$ and assume that \ref{HypoOCP-Gamma} and \ref{HypoOCP-k-Bound} are satisfied. Then, for every $t_0 \in \mathbbm R_+$ and $x_0 \in \mathbbm R^d$, there exists an optimal trajectory $\gamma$ for $(\Gamma, k, t_0, x_0)$.
\end{proposition}

Another property of $\OCP(\Gamma, k)$ than can be obtained by a straightforward argument is the following, which states that restrictions of optimal trajectories are still optimal.

\begin{proposition}
\label{PropRestrictionIsOptimal}
Consider the optimal control problem $\OCP(\Gamma, k)$ and let $(t_0, x_0) \in \mathbbm R_+ \times \mathbbm R^d$ and $\gamma_0 \in \Opt(\Gamma, k, t_0, x_0)$. Then, for every $t_1 \in [t_0, +\infty)$, denoting $x_1 = \gamma_0(t_1)$, the function $\gamma_1: \mathbbm R_+ \to \mathbbm R^d$ defined by $\gamma_1(t) = x_1$ for $t \leq t_1$ and $\gamma_1(t) = \gamma_0(t)$ for $t \geq t_1$ satisfies $\gamma_1 \in \Opt(\Gamma, k, t_1, x_1)$.
\end{proposition}

\subsection{The value function}
\label{SecValueFunction}

We consider in this section properties of the value function corresponding to the optimal control problem $\OCP(\Gamma, k)$, whose definition is given next.
  
\begin{definition}
Let $\Gamma \subset \mathbbm R^d$ be a nonempty closed set and $k: \mathbbm R_+ \times \mathbbm R^d \to \mathbbm R_+$. The \emph{value function} of the optimal control problem $\OCP(\Gamma, k)$ is the function $\varphi: \mathbbm R_+ \times \mathbbm R^d \to \mathbbm R_+$ defined for $(t_0, x_0) \in \mathbbm R_+ \times \mathbbm R^d$ by
\begin{equation}\label{value def}
\varphi(t_0, x_0) = \inf_{\substack{\gamma \in \Adm(k) \\ \gamma(t_0) = x_0}} \tau_\Gamma(t_0,\gamma).
\end{equation}
\end{definition}

Our next preliminary result provides local bounds on the value function and on the norm of optimal trajectories.

\begin{proposition}\label{psi}
Consider the optimal control problem $\OCP(\Gamma, k)$ and its value function $\varphi$ and assume that \ref{HypoOCP-Gamma} and \ref{HypoOCP-k-Bound} are satisfied. Then there exist two nondecreasing maps with linear growth $\psi,\,T:\mathbbm{R_+} \to \mathbbm{R_+}$ depending only on $\Gamma$, $K_{\min}$, and $K_{\max}$ such that, for every $R>0$, $t_0 \in \mathbbm R_+$, $x_0 \in B_R$, we have $\varphi(t_0, x_0) \leq T(R)$ and, for every $\gamma \in \Opt(\Gamma, k, t_0, x_0)$, we have $\gamma(t) \in B_{\psi(R)}$ for every $t \ge 0$.
\end{proposition} 

The bound $T(R)$ on the value function can be obtained, for instance, by remarking that a particular admissible trajectory is the one that moves with speed $K_{\min}$ along the segment from $x_0$ to $0$ and then along the segment from $0$ to the closest point of $\Gamma$ from $0$. Since any optimal trajectory $\gamma$ is $K_{\max}$-Lipschitz and arrives at the target set in time at most $T(R)$, one can easily bound $\abs{\gamma(t)}$ by $\abs{x_0} + K_{\max} T(R)$, yielding the bound on optimal trajectories.

In the next result we recall the dynamic programming principle, which can be proved by standard techniques in optimal control (see, e.g., \cite[Proposition~2.1]{BardiDolcetta} and \cite[(8.4)]{CannarsaPiermarcoSinestrari} for the corresponding result in the autonomous case).

\begin{proposition}\label{dynamic prog-principle}
Consider the optimal control problem $\OCP(\Gamma, k)$ and its value function $\varphi$ and assume that \ref{HypoOCP-Gamma} and \ref{HypoOCP-k-Bound} are satisfied. Then, for every $(t_0, x_0) \in \mathbbm R_+ \times \mathbbm R^d$ and $\gamma \in \Adm(k)$ with $\gamma(t_0) = x_0$, we have
\begin{equation}
\label{eq:DPP}
\varphi(t_0 + h, \gamma(t_0 + h)) + h \geq \varphi(t_0, x_0), \qquad \text{ for every } h \geq 0,
\end{equation}
with equality for every $h \in [0, \tau_\Gamma(t_0, \gamma)]$ if $\gamma \in \Opt(\Gamma, k, t_0, x_0)$. Moreover, if $\gamma$ is constant on $[0, t_0]$ and on $[t_0 + \tau_\Gamma(t_0, \gamma), +\infty)$ and if equality holds in \eqref{eq:DPP} for every $h \in [0, \tau_\Gamma(t_0, \gamma)]$, then $\gamma \in \Opt(\Gamma, k, t_0, x_0)$.
\end{proposition}

Our next preliminary result on $\OCP(\Gamma, k)$ deals with the Lipschitz continuity of the value function. Lipschitz continuity of the value function is a classical result in optimal exit time problems (see, e.g., \cite[Theorem~8.2.5]{CannarsaPiermarcoSinestrari}), but most of the literature deals only with autonomous control systems, in which case the value function is a function of the space variable $x$ only. A classical state augmentation technique of \eqref{optimal-control} would be sufficient to obtain Lipschitz continuity of $\varphi$ on both time and space, but this would require the assumption that $k$ is locally Lipschitz continuous in the pair $(t, x)$, which is stronger than \ref{HypoOCP-k-Lip}. In order to highlight the fact that such an assumption is not necessary, we provide below a detailed proof of the Lipschitz continuity of $\varphi$, based on that of \cite[Theorem~8.2.5]{CannarsaPiermarcoSinestrari} but containing some simplifications due to the particular structure of the problem at hand. We start with a preliminary result stating Lipschitz continuity of $x \mapsto \varphi(t, x)$ for fixed $t \in \mathbbm R_+$.

\begin{lemma}\label{lemm-varphi-Lipschitz}
Consider the optimal control problem $\OCP(\Gamma, k)$ and its value function $\varphi$ and assume that \ref{HypoOCP-Gamma}--\ref{HypoOCP-k-Lip} are satisfied. Then, for every $R > 0$, there exists $C_R > 0$ such that, for every $t_0 \in \mathbbm R_+$ and $x_0,\, x_1 \in B_R$, we have
\[
\abs{\varphi(t_0, x_0) - \varphi(t_0, x_1)} \leq C_R \abs{x_0 - x_1}.
\] 
\end{lemma}

\begin{proof}
Let $T: \mathbbm R_+ \to \mathbbm R_+$ be as in the statement of Proposition~\ref{psi}, $R > 0$, $t_0 \in \mathbbm R_+$, and $x_0, x_1 \in B_R$. Let $\gamma_0 \in \Opt(\Gamma, k, t_0, x_0)$ and denote by $u_0$ the corresponding optimal control, i.e., $\dot\gamma_0(t) = k(t, \gamma_0(t)) u_0(t)$ for a.e.\ $t \in \mathbbm R_+$. Let $t_0^\ast = t_0 + \varphi(t_0, x_0)$ be the time at which $\gamma_0$ arrives at the target set $\Gamma$ and $x_0^\ast = \gamma_0(t_0^\ast) \in \Gamma$ be the arrival position of $\gamma_0$ at $\Gamma$. We define $\gamma_1: \mathbbm R_+ \to \mathbbm R^d$ as follows: for $t \in [0, t_0]$, we set $\gamma_1(t) = x_1$; for $t \in [t_0, t_0^\ast]$, $\gamma_1$ is the unique solution of the differential equation $\dot\gamma_1(t) = k(t, \gamma_1(t)) u_0(t)$ with initial condition $\gamma_1(t_0) = x_1$; for $t \in (t_0^\ast, t_1^\ast]$, we set $\gamma_1(t) = \left(1 - \frac{t - t_0^\ast}{t_1^\ast - t_0^\ast}\right) x_1^\ast + \frac{t - t_0^\ast}{t_1^\ast - t_0^\ast} x_0^\ast$, where $x_1^\ast = \gamma_1(t_0^\ast)$ and $t_1^\ast = t_0^\ast + \frac{\abs{x_1^\ast - x_0^\ast}}{K_{\min}}$; and, for $t > t_1^\ast$, we set $\gamma_1(t) = \gamma_1(t_1^\ast) = x_0^\ast \in \Gamma$. In other words, $\gamma_1$ remains at $x_1$ until time $t_0$, then it is defined as the solution of the control system \eqref{optimal-control} with control $u_0$ until time $t_0^\ast$, and finally $\gamma_1$ moves from its position $x_1^\ast$ at time $t_0^\ast$ to the final position $x_0^\ast$ of $\gamma_0$ along the segment connecting these two points and with constant speed $K_{\min}$, remaining at $x_0^\ast$ afterward. By construction, we have $\gamma_1 \in \Adm(k)$ and $\tau_\Gamma(t_0, \gamma_1) \leq t_1^\ast - t_0 = \varphi(t_0, x_0) + \frac{\abs{x_1^\ast - x_0^\ast}}{K_{\min}}$, and hence
\begin{equation}
\label{eq:varphi-Lipschitz-intermediate-step}
\varphi(t_0, x_1) \leq \varphi(t_0, x_0) + \frac{\abs{x_1^\ast - x_0^\ast}}{K_{\min}}.
\end{equation}

Let us estimate $\abs{x_1^\ast - x_0^\ast}$. Notice first that, since $\gamma_0$ and $\gamma_1$ are $K_{\max}$-Lipschitz, we have, for every $t \in [t_0, t_0^\ast]$ and $i \in \{0, 1\}$,
\[\abs{\gamma_i(t)} \leq \abs{x_i} + K_{\max} (t_0^\ast - t_0) \leq R + K_{\max} T(R).\]
Let $L > 0$ be the Lipschitz constant of $k$ with respect to its second variable on $\mathbbm R_+ \times B_{R + K_{\max} T(R)}$. We then have, for every $t \in [t_0, t_0^\ast]$,
\[
\gamma_1(t) - \gamma_0(t) = x_1 - x_0 + \int_{t_0}^t \left[k(s, \gamma_1(s)) - k(s, \gamma_0(s))\right] u_0(s) \diff s,
\]
and thus
\[
\abs{\gamma_1(t) - \gamma_0(t)} \leq \abs{x_1 - x_0} + L \int_{t_0}^t \abs{\gamma_1(s) - \gamma_0(s)} \diff s.
\]
Hence, by Grönwall's inequality, we deduce that
\[
\abs{x_1^\ast - x_0^\ast} \leq e^{L T(R)} \abs{x_1 - x_0}.
\]
Combining with \eqref{eq:varphi-Lipschitz-intermediate-step}, we obtain that
\[
\varphi(t_0, x_1) \leq \varphi(t_0, x_0) + \frac{e^{L T(R)}}{K_{\min}} \abs{x_1 - x_0}.
\]
The conclusion follows with $C_R = \frac{e^{L T(R)}}{K_{\min}}$ by exchanging the role of $x_0$ and $x_1$ in the above argument.
\end{proof}

We can now deduce Lipschitz continuity of $\varphi$ by using Lemma~\ref{lemm-varphi-Lipschitz} and the dynamic programming principle from Proposition~\ref{dynamic prog-principle}.

\begin{proposition}\label{varphi is Lipschitz}
Consider the optimal control problem $\OCP(\Gamma, k)$ and its value function $\varphi$ and assume that \ref{HypoOCP-Gamma}--\ref{HypoOCP-k-Lip} are satisfied. Then, for every $R > 0$, there exists $M_R > 0$ such that, for every $(t_0, x_0), (t_1, x_1) \in \mathbbm R_+ \times B_R$, we have
\[
\abs{\varphi(t_0, x_0) - \varphi(t_1, x_1)} \leq M_R \left(\abs{t_0 - t_1} + \abs{x_0 - x_1}\right).
\] 
\end{proposition}

\begin{proof}
Let $\psi: \mathbbm R_+ \to \mathbbm R_+$ be as in the statement of Proposition~\ref{psi}, $R > 0$, and $(t_0, x_0),\, (t_1, x_1) \in \mathbbm R_+ \times B_R$ and assume, with no loss of generality, that $t_0 < t_1$. Let $\gamma_0 \in \Opt(\Gamma, k, t_0, x_0)$ and $x_0^\ast = \gamma_0(t_1)$. By Proposition~\ref{psi}, we have $\abs{x_0^\ast} \leq \psi(R)$ and, by Lemma~\ref{lemm-varphi-Lipschitz}, we have
\begin{equation}
\label{eq:varphi-Lipschitz-intermediate-step-2}
\abs{\varphi(t_1, x_0^\ast) - \varphi(t_1, x_1)} \leq C_{\psi(R)} \abs{x_0^\ast - x_1},
\end{equation}
where $C_{\psi(R)}$ denotes the Lipschitz constant of $x \mapsto \varphi(t, x)$ on $B_{\psi(R)}$ for all $t \geq 0$.

If $t_1 \leq t_0 + \varphi(t_0, x_0)$, then, by Proposition~\ref{dynamic prog-principle}, since $\gamma_0 \in \Opt(\Gamma, k, t_0, x_0)$, we have $\varphi(t_1, x_0^\ast) = \varphi(t_0, x_0) - (t_1 - t_0)$, and thus
\begin{equation}
\label{eq:varphi-Lipschitz-intermediate-step-3}
\abs{\varphi(t_0, x_0) - \varphi(t_1, x_1)} \leq \abs{t_1 - t_0} + C_{\psi(R)} \abs{x_0^\ast - x_1}.
\end{equation}
Otherwise, we have $t_1 > t_0 + \varphi(t_0, x_0)$, in which case $x_0^\ast = \gamma_0(t_1) = \gamma_0(t_0 + \varphi(t_0, x_0)) \in \Gamma$ and thus $\varphi(t_1, x_0^\ast) = 0$. Combining this with \eqref{eq:varphi-Lipschitz-intermediate-step-2} and the fact that $\varphi(t_0, x_0) < t_1 - t_0$, we deduce that \eqref{eq:varphi-Lipschitz-intermediate-step-3} also holds in this case.

Since $\gamma_0$ is $K_{\max}$-Lipschitz, we have $\abs{x_0 - x_0^\ast} \leq K_{\max} \abs{t_1 - t_0}$. Hence, combining with \eqref{eq:varphi-Lipschitz-intermediate-step-3}, we deduce that
\[
\abs{\varphi(t_0, x_0) - \varphi(t_1, x_1)} \leq (C_{\psi(R)} K_{\max} + 1) \abs{t_1 - t_0} + C_{\psi(R)} \abs{x_0 - x_1},
\]
yielding the conclusion.
\end{proof}

A classical consequence of the dynamic programming principle is that the value function $\varphi$ satisfies a Hamilton--Jacobi equation in the viscosity sense, which is the topic of the next proposition, whose proof is omitted here since it can be obtained by adapting classical arguments (see, e.g., \cite[Chapter~IV, Proposition~2.3]{BardiDolcetta} and \cite[Theorem~8.1.8]{CannarsaPiermarcoSinestrari}) to our non-autonomous setting.

\begin{proposition}\label{thm H-J}
Consider the optimal control problem $\OCP(\Gamma, k)$ and its value function $\varphi$ and assume that \ref{HypoOCP-Gamma}--\ref{HypoOCP-k-Lip} are satisfied. Consider the Hamilton--Jacobi equation
\begin{equation}\label{H-J equation}
-\partial_t \varphi(t, x) + k(t, x) \abs*{\nabla \varphi(t, x)} - 1 = 0.
\end{equation}
Then $\varphi$ is a viscosity solution of \eqref{H-J equation} on $\mathbbm R_+ \times (\mathbbm R^d \setminus \Gamma)$ and satisfies $\varphi(t, x) = 0$ for $(t, x) \in \mathbbm R_+ \times \Gamma$.

\end{proposition}

We next provide the following property of $\varphi$, whose proof can be found in \cite[Proposition~3.9 and Corollary~3.11]{Dweik2020Sharp}.

\begin{proposition}\label{PropLowerBoundPartialT}
Consider the optimal control problem $\OCP(\Gamma, k)$ and its value function $\varphi$ and assume that \ref{HypoOCP-Gamma}--\ref{HypoOCP-k-Lip} are satisfied. Then, for every $R > 0$, there exists $c > 0$ such that, for every $t_0, t_1 \in \mathbbm R_+$ with $t_0 \neq t_1$ and $x \in B_R$, we have
\[
\frac{\varphi(t_1, x) - \varphi(t_0, x)}{t_1 - t_0} \geq c - 1.
\] 
In particular, if $\varphi$ is differentiable at $(t_0, x)$, then $\partial_t \varphi(t_0, x) \geq c - 1$ and $\abs{\nabla\varphi(t_0, x)} \geq \frac{c}{K_{\max}}$.
\end{proposition}

\subsection{Characterization of optimal controls}

Now that we have established elementary properties of the value function in Section~\ref{SecValueFunction}, we turn to the problem of characterizing the optimal control $u: \mathbbm R_+ \to B_1$ associated with an optimal trajectory $\gamma \in \Opt(\Gamma, k, t_0, x_0)$. Formally, by differentiating with respect to $h$ the equality of the dynamic programming principle in Proposition~\ref{dynamic prog-principle} for optimal trajectories and using the Hamilton--Jacobi equation \eqref{H-J equation}, one obtains that the optimal control $u$ should satisfy $u(t) = -\frac{\nabla\varphi(t, \gamma(t))}{\abs{\nabla\varphi(t, \gamma(t))}}$, an argument that can be made precise when $\varphi$ is differentiable at $(t, \gamma(t))$ (see, e.g., \cite[Corollary~4.1]{Mazanti2019Minimal}).

If $\varphi$ was semiconcave, one could deduce by standard arguments (see, e.g., \cite[Section~7.3]{CannarsaPiermarcoSinestrari} and \cite[Section~3.4]{Dweik2020Sharp}) that it is differentiable along optimal trajectories and hence obtain the above characterization of optimal controls. In particular (see, e.g., \cite[Theorem 7.3.16]{CannarsaPiermarcoSinestrari}), $\varphi$ can be shown to be semiconcave under the additional assumption that $k \in \mathbf C^{1, 1}(\mathbbm R_+ \times \mathbbm R^d; \mathbbm R)$ (i.e., $k$ is $\mathbf C^1$ and its differential is locally Lipschitz continuous). On the other hand, under our standing assumptions \ref{HypoOCP-Gamma}--\ref{HypoOCP-k-Lip}, neither semiconcavity nor differentiability of $\varphi$ along optimal trajectories are guaranteed, and, up to the authors' knowledge, it is an open question if these properties hold or not.

The goal of this section is to provide an alternative characterization of $u$ when $k$ is not necessarily more regular than locally Lipschitz continuous. This is done mainly for two reasons. Firstly, regularity assumptions on $k$ for $\OCP(\Gamma, k)$ correspond to regularity assumptions on $K_i$, $i \in \{1, \dotsc, N\}$, for $\MFG(\mathbf\Gamma, \mathbf K, \mathbf{m_0})$, and hence avoiding additional regularity assumptions on $k$ allow to obtain more general results for mean field games. Secondly, even when $k$ is smooth, the value function $\varphi$ may fail to be semiconcave in some situations, such as in the presence of state constraints (see, e.g., \cite[Example~4.4]{Cannarsa2008Lipschitz}), and semiconcavity of $\varphi$ is a key step in proving its differentiability along optimal trajectories and hence in characterizing $u$ as above. This motivates the search for techniques for characterizing optimal controls without relying on the semiconcavity of $\varphi$.

We shall need in this section the following additional assumption on $k$.

\begin{hypotheses}[resume]
\item\label{HypoOCP-k-loc-Lip} The function $k: \mathbbm R_+ \times \mathbbm R^d \to \mathbbm R_+$ is Lipschitz continuous with respect to both variables and locally in the second variable, i.e., for every $R > 0$, there exists $L > 0$ such that, for every $(t_1, x_1),\, (t_2, x_2) \in \mathbbm R_+ \times B_R$, we have
\[
\abs{k(t_1, x_1) - k(t_2, x_2)} \leq L \left(\abs{t_1 - t_2} + \abs{x_1 - x_2}\right).
\]
\end{hypotheses}

The first result we present is the following, which provides additional regularity assumptions on the optimal control $u$. It can be obtained by applying Pontryagin Maximum Principle to $\OCP(\Gamma, k)$ and using the maximization condition to deduce a relation between the optimal control $u$ and the costate variable in Pontryagin Maximum Principle. We refer the reader to \cite[Proposition~4.6 and Corollary~4.2]{Mazanti2019Minimal} for the details of the proof.

\begin{proposition}\label{prop smooth 1}
Consider the optimal control problem $\OCP(\Gamma, k)$ and assume that \ref{HypoOCP-Gamma}, \ref{HypoOCP-k-Bound}, and \ref{HypoOCP-k-loc-Lip} hold. Let $(t_0, x_0) \in \mathbbm R_+ \times \mathbbm R^d$, $\gamma \in \Opt(\Gamma, k, t_0, x_0)$, and $u$ be the optimal control corresponding to $\gamma$. Then $u \in \Lip([t_0, t_0+\varphi(t_0, x_0)]; \mathbbm S^{d-1})$. Moreover, its Lipschitz constant is bounded by the Lipschitz constant of $k$ on the set $[t_0, t_0 + \varphi(t_0, x_0)] \times B_R$, where $R > 0$  is such that $\gamma(t) \in B_R$ for every $t \geq 0$.
\end{proposition}

We now introduce the two main objects that we will use to characterize optimal controls.

\begin{definition}
\label{DefUW}
Consider the optimal control problem $\OCP(\Gamma, k)$ and its value function $\varphi$ and assume that \ref{HypoOCP-Gamma}, \ref{HypoOCP-k-Bound}, and \ref{HypoOCP-k-loc-Lip} hold. Let $(t_0, x_0) \in \mathbbm R_+ \times \mathbbm R^d$.
\begin{enumerate}
\item We define the set $\mathcal U(t_0, x_0)$ of \emph{optimal directions} at $(t_0, x_0)$ as the set of all $u_0 \in \mathbbm S^{d-1}$ for which there exists $\gamma \in \Opt(\Gamma, k, t_0, x_0)$ such that the corresponding optimal control $u$ satisfies $u(t_0) = u_0$.
\item We define the set $\mathcal W(t_0, x_0)$ of \emph{directions of maximal descent of $\varphi$} at $(t_0, x_0)$ as the set of all $u_0 \in \mathbbm S^{d-1}$ such that
\begin{equation}
\label{eq:limit-def-W}
\lim_{h\to 0^+}\frac{\varphi(t_0 + h, x_{0} + h k(t_0, x_{0}) u_0) - \varphi(t_0, x_{0})}{h} = -1.
\end{equation}
\end{enumerate}
\end{definition}

Thanks to Proposition~\ref{prop smooth 1}, optimal controls are continuous and take values in $\mathbbm S^{d-1}$, and in particular the pointwise value $u(t_0)$ is well-defined. Together with Proposition~\ref{PropExistOpt}, we immediately deduce that $\mathcal U(t_0, x_0) \neq \varnothing$ for every $(t_0, x_0) \in \mathbbm R_+ \times (\mathbbm R^d \setminus \Gamma)$. On the other hand, for $(t_0,x_0)\in \mathbbm R^d\times \Gamma$, one observes that $\mathcal U(t_0, x_0)= \varnothing$, since, when $x_0 \in \Gamma$, the only optimal control is the control constantly equal to $0$, but, by definition, the members of $\mathcal U(t_0, x_0)$ must belong to the unit sphere $\mathbbm S^{d-1}$.

Note also that, if $u_0 \in \mathbbm S^{d-1}$ and $\gamma \in \Adm(k)$ is the trajectory obtained by taking a constant control $u(t) = u_0$ in \eqref{optimal-control}, then, by Proposition~\ref{dynamic prog-principle}, $\varphi(t_0 + h, \gamma(t_0 + h)) - \varphi(t_0, x_0) \geq -h$, yielding, using also Proposition~\ref{varphi is Lipschitz}, that, as $h \to 0^+$,
\[
\frac{\varphi(t_0 + h, x_{0} + h k(t_0, x_{0}) u_0) - \varphi(t_0, x_{0})}{h} \geq -1 + o(1).
\]
Hence, an element $u_0 \in \mathcal W(t_0, x_0)$ can be interpreted as a direction in which the above ratio attains its infinitesimal lower bound $-1$ at the limit $h \to 0^+$, and corresponds thus to directions in which $\varphi$ decreases with maximal rate.

Before turning to the main result of this section, Theorem~\ref{thm Ut_0,x_0}, asserting the equality between $\mathcal U(t_0, x_0)$ and $\mathcal W(t_0, x_0)$, let us first present some elementary properties of these set-valued maps. The first one is that, along an optimal trajectory $\gamma$, $\mathcal{U}(t, \gamma(t))$ is singleton, except possibly at its initial and final points. Its proof is the same as that of \cite[Proposition~4.7]{Mazanti2019Minimal} and is thus omitted here.

\begin{proposition}
\label{PropUSingleElement}
Consider the optimal control problem $\OCP(\Gamma, k)$ and its value function $\varphi$ and assume that \ref{HypoOCP-Gamma}, \ref{HypoOCP-k-Bound}, and \ref{HypoOCP-k-loc-Lip} hold. Let $(t_0, x_0) \in \mathbbm R_+ \times \mathbbm R^d$ and $\gamma \in \Opt(\Gamma, k, t_0, x_0)$. Then, for every $t \in (t_0, t_0 + \varphi(t_0, x_0))$, $\mathcal{U}(t,\gamma(t))$ contains exactly one element.
\end{proposition}

Our next result shows, on the other hand, that, at the points $(t_0, x_0)$ where $\varphi$ is differentiable, $\mathcal W(t_0, x_0)$ contains a unique direction of maximal descent which, as one might expect, is equal to $-\frac{\nabla\varphi(t_0, x_0)}{\abs{\nabla\varphi(t_0, x_0)}}$, as $\abs{\nabla\varphi(t_0, x_0)} \neq 0$ is guaranteed by Proposition~\ref{PropLowerBoundPartialT}.

\begin{proposition}
\label{PropWNormalizedGrad}
Consider the optimal control problem $\OCP(\Gamma, k)$ and its value function $\varphi$ and assume that \ref{HypoOCP-Gamma}, \ref{HypoOCP-k-Bound}, and \ref{HypoOCP-k-loc-Lip} hold. Let $(t_0, x_0) \in \mathbbm R_+ \times (\mathbbm R^d \setminus \Gamma)$ be such that $\varphi$ is differentiable at $(t_0, x_0)$. Then
$$
\mathcal{W}(t_0, x_0)= \left\{-\frac{\nabla \varphi(t_0, x_0)}{\abs*{\nabla \varphi(t_0, x_0)}} \right\}.
$$ 
\end{proposition}

\begin{proof}
Since $\varphi$ is differentiable at $(t_0, x_0)$ and using Proposition~\ref{thm H-J}, we have, for every $u_0 \in \mathbbm S^{d-1}$,
\begin{multline*}
\lim_{h\to 0^+}\frac{\varphi(t_0 + h, x_{0} + h k(t_0, x_{0}) u_0) - \varphi(t_0, x_{0})}{h} \\ = \partial_t \varphi(t_0, x_0) + k(t_0, x_0) \nabla\varphi(t_0, x_0) \cdot u_0 = -1 + k(t_0, x_0) [\nabla\varphi(t_0, x_0) \cdot u_0 + \abs{\nabla\varphi(t_0, x_0)}].
\end{multline*}
Hence \eqref{eq:limit-def-W} holds if and only if $\nabla\varphi(t_0, x_0) \cdot u_0 = - \abs{\nabla\varphi(t_0, x_0)}$ and, since $\nabla\varphi(t_0, x_0) \neq 0$ by Proposition~\ref{PropLowerBoundPartialT}, it follows that \eqref{eq:limit-def-W} holds if and only if $u_0 = -\frac{\nabla \varphi(t_0, x_0)}{\abs*{\nabla \varphi(t_0, x_0)}}$, yielding the conclusion.
\end{proof}

The main result of this section is the following.

\begin{theorem}\label{thm Ut_0,x_0}
Consider the optimal control problem $\OCP(\Gamma, k)$ and its value function $\varphi$ and assume that \ref{HypoOCP-Gamma}, \ref{HypoOCP-k-Bound}, and \ref{HypoOCP-k-loc-Lip} hold. Then, for every $(t_0, x_0) \in \mathbbm R_+ \times \mathbbm R^d$, we have $\mathcal{U}(t_0, x_0) = \mathcal{W}(t_0, x_{0})$.
\end{theorem}

\begin{proof}
We first remark that, if $x_0 \in \Gamma$, then $\mathcal{U}(t_0, x_0) = \mathcal{W}(t_0, x_{0}) = \varnothing$, and so we are only left to consider the case $x_0 \in \mathbbm R^d \setminus \Gamma$.

The inclusion $\mathcal{U}(t_0, x_{0}) \subset \mathcal{W}(t_0, x_{0})$ follows from the fact that, if $\gamma \in \Opt(\Gamma, k, t_0, x_0)$ and $u$ is the corresponding optimal control, then, by Proposition~\ref{dynamic prog-principle}, we have, for every $h \in (0, \varphi(t_0, x_0)]$, that
\[\frac{\varphi(t_0 + h, \gamma(t_0 + h)) - \varphi(t_0, x_0)}{h} = -1\]
and, using the facts that $\gamma(t_0 + h) = x_0 + h k(t_0, x_0) u(t_0) + o(h)$ and that $\varphi$ is locally Lipschitz continuous (Proposition~\ref{varphi is Lipschitz}), we deduce, letting $h \to 0^+$, that $u(t_0) \in \mathcal W(t_0, x_0)$.

Let us now show that $\mathcal W(t_0, x_0) \subset \mathcal U(t_0, x_0)$. 
Let $u_{0} \in \mathcal{W}(t_0, x_{0})$ and $h > 0$, which is implicitly always assumed to be small enough. Then, as $h \to 0^+$,
\begin{equation}
\label{eq:u0-in-W}
\varphi(t_0 + h, x_{0} + h k(t_0, x_{0}) u_{0}) = \varphi(t_0, x_{0}) - h + o(h).
\end{equation}
Define $\gamma_{0}: [t_0, t_0 + h] \to \mathbbm R^d$ by 
\begin{equation}\label{ODE cont1}
    \left \{
    \begin{aligned}
    \Dot{\gamma}_{0}(t) & = k(t,\gamma_{0}(t)) u_{0},
    \\  
    \gamma_{0}(t_0) & = x_{0}.       
    \end{aligned} \right.
\end{equation}
Let $x_{1}^h = \gamma_{0}(t_0 + h)$ and $t_1^h = t_0 + h$. Since $\mathbbm R^d \setminus \Gamma$ is open, one has $x_1^h \in \mathbbm R^d \setminus \Gamma$ for $h > 0$ small enough. Let $\gamma_{1}^h \in \Opt(\Gamma, k, t_1^h, x_{1}^h)$ and $u_{1}^h$ be the optimal control associated with $\gamma_{1}^h$. Set $\Bar{u}_{1}^h = u_{1}^h(t_1^h)\in \mathbbm S^{d-1}$ and define $\Bar\gamma_1^h: [t_1^h, t_1^h + h] \to \mathbbm R^d$ by
\begin{equation}\label{ODE cont2}
    \left \{
    \begin{aligned}
    \Dot{\Bar{\gamma}}_{1}^h(t) & = k(t, \Bar{\gamma}_{1}^h(t)) \Bar{u}_{1}^h \\  
    \Bar{\gamma}_{1}^h(t_1^h) & = x_{1}^h.
    \end{aligned} \right.
\end{equation}
Let us also set $t_2^h=t_1^h+h$, $x_{2}^h = \gamma_{1}^h(t_2^h)$ and $\Bar{x}_{2}^h = \Bar{\gamma}_{1}^h(t_2^h)$. We split the sequel of the proof in two cases.

\medskip

\noindent\textbf{Case 1.} We assume in this case that $\lim_{h \to 0^+} \Bar{u}_{1}^h = u_{0}$. Let $\hat u_1^h \in \Lip(\mathbbm R_+; \mathbbm S^{d-1})$ be defined by $\hat u_1^h(t) = \bar u_1^h$ for $t \in [0, t_1^h]$, $\hat u_1^h(t) = u_1^h(t)$ for $t \in [t_1^h, t_1^h + \varphi(t_1^h, x_1^h)]$, and $\hat u_1^h(t) = u_1^h(t_1^h + \varphi(t_1^h, x_1^h))$ for $t \geq t_1^h + \varphi(t_1^h, x_1^h)$. Since $\gamma_{1}^h$ and $\hat u_{1}^h$ are Lipschitz continuous and their Lipschitz constants do not depend on $h$ (see Proposition~\ref{prop smooth 1}), one deduces from Arzelà--Ascoli Theorem that there exist a positive sequence $(h_n)_{n \in \mathbbm N}$ converging to $0$ as $n \to +\infty$ and elements $\gamma^* \in \Lip_{K_{\max}}(\mathbbm R_+; \mathbbm R^d)$ and $u^* \in \Lip(\mathbbm R_+; \mathbbm S^{d-1})$ such that $\gamma_{1}^{h_n} \to \gamma^*$ and $\hat u_{1}^{h_n} \to u^*$ as $n \to +\infty$, uniformly on compact time intervals. Since $\gamma_1^h \in \Opt(\Gamma, k, t_1^h, x_1^h)$ for $h > 0$ and $t_1^h \to t_0$ and $x_1^h \to x_0$ as $h \to 0^+$, one can easily show, using the continuity of $\varphi$, that $\gamma^* \in \Opt(\Gamma, k, t_0, x_0)$ and the restriction of $u^*$ to $[t_0, t_0 + \varphi(t_0, x_0)]$ is its corresponding optimal control. On the other hand, we have
$$
u^*(t_0) = \lim_{n \to +\infty} \hat u_{1}^{h_n} (t_1^{h_n}) = \lim_{n \to +\infty} \Bar{u}_{1}^{h_n} = u_{0},
$$
which implies that $u_0 \in \mathcal U(t_0, x_0)$, as required.

\medskip

\noindent\textbf{Case 2.} We now consider the case where $(\bar u_1^h)_{h > 0}$ does not converge to $u_0$ as $h \to 0^+$, and we prove that this case is not possible. Let $\epsilon > 0$ and $(h_n)_{n\in \mathbbm N}$ be a positive sequence such that $h_n \to 0$ as $n\to +\infty$ and $\abs{\Bar{u}_{1}^{h_n}-u_{0}} \ge \epsilon$ for every $n\in \mathbbm N$. For simplicity, we set $t_1^{h_n}=t_1^n$, $x_{1}^{h_n}=x_{1}^n$, and similarly for all other variables whose upper index is $h_n$. In order to clarify the constructions used in this case, we illustrate them in Figure~\ref{FigCase2}, which represents points and curves already constructed as well as those which will be defined in the sequel of the proof.

\tikzset{
 mid arrow/.style={postaction={decorate,decoration={
        markings,
        mark=at position .5 with {\arrow{Stealth}}
      }}}
}

\begin{figure}[ht]
\centering
\begin{tikzpicture}
\node (x0) at (0, 0) {};
\node (x1) at (0, -2) {};
\node (x2bar) at (3, -2) {};
\node (x2) at (3.25, -3.25) {};

\draw[red, mid arrow] (x0.center) -- node[midway, left] {$\gamma_0$} (x1.center);
\draw[blue, mid arrow] (x0.center) -- node[midway, above right] {$\gamma_2^n$} (x2bar.center);
\draw[violet, mid arrow] (x1.center) to[out = 0, in = 180] node[midway, below left] {$\gamma_1^n$} (x2.center);
\draw[violet, mid arrow] (x2.center) -- ++(1, 0);
\draw[green!50!black, mid arrow] (x1.center) -- node[midway, above] {$\bar \gamma_1^n$} (x2bar.center);
\draw[orange!75!black, mid arrow] (x2bar.center) -- node[midway, right] {$\gamma_3^n$} (x2.center);

\fill (x0) circle[radius=0.05] node[left] {$x_0$};
\fill (x1) circle[radius=0.05] node[left] {$x_1^n$};
\fill (x2bar) circle[radius=0.05] node[right] {$\bar x_2^n$};
\fill (x2) circle[radius=0.05] node[below] {$x_2^n$};
\fill (2, -4/3) circle[radius=0.05] node[above right] {$x_3^n$};
\end{tikzpicture}
\caption{Illustration of the constructions used in the proof of Theorem~\ref{thm Ut_0,x_0}.}
\label{FigCase2}
\end{figure}
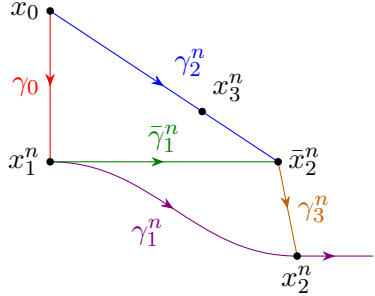

Integrating \eqref{ODE cont1} on $[t_0, t_1^n]$, we get
$$
x_{1}^n-x_{0} = \int_{t_0}^{t_1^n} k(s, \gamma_{0}(s)) \diff s\, u_{0},
$$
and, proceeding similarly for \eqref{ODE cont2}, we get
$$
\Bar{x}_{2}^n- x_{1}^n = \int_{t_1^n}^{t_2^n} k(s, \Bar{\gamma}_{1}^n(s)) \diff s\, \Bar{u}_{1}^n.
$$
Denote the integrals in the right-hand side of the above equalities by $I_0^n$ and $I_1^n$, respectively. We have
\begin{equation*}
    \begin{aligned}
    \abs{\Bar{x}_{2}^n-x_{0}}^2 & = (I_0^n u_{0} + I_1^n \Bar{u}_{1}^n)\cdot (I_0^n u_{0} + I_1^n \Bar{u}_{1}^n)\\
    & = (I_0^n)^2 + (I_1^n)^2 + 2I_0^n I_1^n u_{0} \cdot \Bar{u}_{1}^n \\
    &=\abs{x_{1}^n-x_{0}}^2 + \abs*{\Bar{x}_{2}^n - x_{1}^n}^2 + 2 I_0^n I_1^n u_{0} \cdot \Bar{u}_{1}^n.
    \end{aligned}
\end{equation*}
We know that $\abs{\Bar{u}_{1}^{n}-u_{0}} \ge \epsilon$, which leads us to observe that there exists $\alpha \in (0, 1)$ such that $u_{0} \cdot \Bar{u}_{1}^n < \alpha$ for every $n\in \mathbbm N$. Thus
$$
\abs{\Bar{x}_{2}^n - x_{0}}^2 < \abs{x_{1}^n-x_{0}}^2 + \abs*{\Bar{x}_{2}^n - x_{1}^n}^2 + 2 \alpha I_0^n I_1^n.
$$
Define
$$
\rho := \sqrt{1 - (1-\alpha) \frac{K_{\min}^2}{2 K_{\max}^2}},
$$
then obviously $\rho<1$ and
\begin{align}
\abs{\Bar{x}_{2}^n - x_{0}}^2 & < \left(\abs{x_{1}^n-x_{0}} + \abs*{\Bar{x}_{2}^n - x_{1}^n}\right)^2 - 2 (1 - \alpha) I_0^n I_1^n \notag \\
 & = \left(1 - (1 - \alpha) \frac{2 I_0^n I_1^n}{\left(I_0^n + I_1^n\right)^2}\right)\left(\abs{x_{1}^n-x_{0}} + \abs*{\Bar{x}_{2}^n - x_{1}^n}\right)^2 \notag \\
 & \leq \rho^2 \left(\abs*{ x_{1}^n - x_{0}} + \abs*{\Bar{x}_{2}^n - x_{1}^n}\right)^2, \label{eq:dist-x2-x0}
\end{align}
where we use that $I_i^n \in [h K_{\min}, h K_{\max}]$ for $i \in \{1, 2\}$. Let $u_{2}^n = \frac{\Bar{x}_{2}^n - x_{0}}{\abs{\Bar{x}_{2}^n - x_{0}}}$ (with the convention $u_2^n = 0$ if $\bar x_2^n = x_0$) and define $\gamma_{2}^n: [t_0, t_0 + \tau^n] \to \mathbbm R^d$ by
\begin{equation}\label{ODE cont3}
    \left \{
    \begin{aligned}
    \Dot{\gamma}_{2}^n(t) & = k(t, \gamma_{2}^n(t)) u_{2}^n \\  
    \gamma_{2}^n(t_0) & = x_{0},       
    \end{aligned}
		\right.
\end{equation}
where $\tau^n \geq 0$ is chosen so that $\gamma_2^n(t_0 + \tau^n) = \bar x_2^n$.

\begin{claim}
As $n \to +\infty$, we have $\tau^n \leq 2 \rho h_n + O(h_n^2)$.
\end{claim}

\begin{proof}
Note that we have nothing to prove in the case $\bar x_2^n = x_0$, and hence we assume $\bar x_2^n \neq x_0$ in the sequel. If $\abs{\bar x_2^n - x_0} \leq \rho \abs{x_1^n - x_0}$, we let $x_3^n = \bar x_2^n$, otherwise we choose $x_{3}^n$ as the unique point in the segment $(x_{0}, \bar x_{2}^n)$ such that $\abs{x_{3}^n - x_{0}} = \rho \abs{x_{1}^n - x_{0}}$. In both cases, we have $\abs{x_{3}^n - x_{0}} = \bar\rho \abs{x_{1}^n - x_{0}}$ for some $\bar \rho \leq \rho$. Let $\tau_1^n$ be the time that $\gamma_{2}^n$ takes to reach the point $x_{3}^n$, i.e., $\gamma_{2}^n (t_0 + \tau_1^n) = x_{3}^n$. (Note that $\tau_1^n = \tau^n$ in the case $\abs{\bar x_2^n - x_0} \leq \rho \abs{x_1^n - x_0}$.) We show that $\tau_1^n \le \rho h_n + O(h_n^2)$. To obtain that, we observe, by integrating \eqref{ODE cont1} and \eqref{ODE cont3} and doing a change of variables, that
    \begin{align}
    \int_{t_0}^{t_0+\tau_1^n}k(s, \gamma_{2}^n(s)) \diff s & = \abs{x_3^n - x_0} = \bar\rho \abs{x_1^n - x_0} = \bar\rho \int_{t_0}^{t_0 + h_n} k(s, \gamma_0(s)) \diff s \notag \\
		& = \int_{t_0}^{t_0 + \bar\rho h_n} k\left(t_0 + \frac{s - t_0}{\bar\rho}, \gamma_0\left(t_0 + \frac{s - t_0}{\bar\rho}\right)\right) \diff s \notag \\
		& = \int_{t_0}^{t_0+\bar\rho h_n} k(s, \gamma_{2}^n (s))\diff s \notag \\
    & \hphantom{{} = {} } + \int_{t_0}^{t_0+\bar\rho h_n} \left[k\left(t_0 + \frac{s - t_0}{\bar\rho}, \gamma_{0}\left(t_0+\frac{s-t_0}{\bar\rho}\right)\right) - k(s,\gamma_{2}^n(s))\right] \diff s \notag \\
    & = \int_{t_0}^{t_0+\bar\rho h_n} k(s, \gamma_{2}^n(s)) \diff s + O(h_n^2), \label{eq:k-gamma2}
    \end{align}
in which the last equality follows from the Lipschitz continuity of $k$ and the fact that
\begin{align*}
\abs*{\gamma_{0}\left(t_0+\frac{s-t_0}{\bar\rho}\right) - \gamma_{2}^n(s)} & \leq \abs*{\gamma_{0}\left(t_0+\frac{s-t_0}{\bar\rho}\right) - x_0} + \abs*{x_0 - \gamma_{2}^n(s)} \\
 & \leq K_{\max} \left[\frac{s - t_0}{\bar\rho} + (s - t_0)\right].
\end{align*}
Define $F: [0, \tau^n] \to \mathbbm R_+$ by $F(t)=\int_{t_0}^{t_0 + t} k(s,\gamma_{2}^n(s)) \diff s$, then obviously $F$ is increasing, which implies that $F^{-1}$ is well-defined on the range of $F$. Since $F^{\prime}(t) = k(t, \gamma_{2}^n(t))$, $F$ is $K_{\max}$-Lipschitz continuous and, since $(F^{-1})^{\prime}(t)=\frac{1}{F^{\prime}(F^{-1}(t))}$, we also deduce that $F^{-1}$ is $\frac{1}{K_{\min}}$-Lipschitz continuous.
Therefore, by \eqref{eq:k-gamma2}, we deduce that
\begin{equation*}
\tau_1^n = F^{-1}(F(\bar\rho h_n) + O(h_n^2)) = \bar\rho h_n + O(h_n^2) \leq \rho h_n + O(h_n^2).
\end{equation*}
This concludes the proof of the claim in the case $\abs{\bar x_2^n - x_0} \leq \rho \abs{x_1^n - x_0}$, since $\tau_1^n = \tau^n$ in that case.

Otherwise, we have $\bar\rho = \rho$ and $\abs{x_{3}^n - x_{0}} = \rho \abs{x_{1}^n - x_{0}}$, and thus, from \eqref{eq:dist-x2-x0}, we get
$$
\abs*{\Bar{x}_{2}^n-x_{0}} < \rho (\abs*{ x_{1}^n-x_{0}}+\abs*{\Bar{x}_{2}^n- x_{1}^n}) = \abs*{x_{3}^n - x_{0}} + \rho\abs*{\Bar{x}_{2}^n- x_{1}^n}.
$$
On the other hand, since $x_{3}^n$ belongs to the segment $(x_0, \bar x_2^n)$, we have $\abs*{\Bar{x}_{2}^n - x_{0}} = \abs*{\Bar{x}_{2}^n - x_{3}^n}\allowbreak +\abs*{x_{3}^n - x_{0}}$, hence the inequality $\abs*{\Bar{x}_{2}^n - x_{3}^n} \le \rho \abs*{\Bar{x}_{2}^n- x_{1}^n}$ holds. Suppose $\tau_2^n$ is the time the trajectory $\gamma_{2}^n$ takes to go from $x_{3}^n$ to $\Bar{x}_{2}^n$, i.e., $\gamma_{2}^n(t_0 + \tau_1^n + \tau_2^n) = \Bar{x}_{2}^n$, and note that $\tau^n = \tau_1^n + \tau_2^n$. As before, we compare the times between $\abs*{ \Bar{x}_{2}^n - x_{3}^n}$ and $\abs*{ \Bar{x}_{2}^n- x_{1}^n}$. Let $\beta \leq \rho$ be such that $\abs*{\Bar{x}_{2}^n - x_{3}^n}=\beta \abs*{ \Bar{x}_{2}^n- x_{1}^n}$. Proceeding similarly to \eqref{eq:k-gamma2}, we get
\begin{equation*}
    \begin{aligned}
    \int_{0}^{\tau_2^n}&k(s+t_0+\tau_1^n,\gamma_{2}^n(s+t_0+\tau_1^n)) \diff s = \abs*{\Bar{x}_{2}^n - x_{3}^n}=\beta \abs*{ \Bar{x}_{2}^n- x_{1}^n} = \beta \int_{t_1^n}^{t_2^n} k(s, \bar\gamma_1^n(s)) \diff s \\
		& = \int_{0}^{\beta h_n} k\left(\frac{s}{\beta}+t_0+h_n,\Bar{\gamma}_{1}^n\left(\frac{s}{\beta}+t_0+h_n\right)\right) \diff s \\
		& = \int_{0}^{\beta h_n} k(s+t_0+\tau_1^n,\gamma_{2}^n(s+t_0+\tau_1^n)) \diff s\\
    & \hphantom{{} = {}} +\int_{0}^{\beta h_n} \left[k\left(\frac{s}{\beta}+t_0+h_n,\Bar{\gamma}_{1}^n\left(\frac{s}{\beta}+t_0+h_n\right)\right)-k(s+t_0+\tau_1^n,\gamma_{2}^n(s+t_0+\tau_1^n))\right] \diff s\\
    &=\int_{0}^{\beta h_n} k(s+t_0+\tau_1^n,\gamma_{2}^n(s+t_0+\tau_1^n)) \diff s+O(h_n^2),
    \end{aligned}
\end{equation*}
in which the last equality follows from the Lipschitz continuity of $k$ and the facts that $\tau_1^n = O(h_n)$ and
\begin{multline*}
\abs*{\Bar{\gamma}_{1}^n\left(\frac{s}{\beta}+t_0+h_n\right) - \gamma_{2}^n(s+t_0+\tau_1^n))} \\
\leq \abs*{\Bar{\gamma}_{1}^n\left(\frac{s}{\beta}+t_0+h_n\right) - x_1^n} + \abs{x_1^n - x_0} + \abs*{x_0 - \gamma_{2}^n(s+t_0+\tau_1^n))} \\
 \leq K_{\max} \left[\frac{s}{\beta} + h_n + s + \tau_1^n\right].
\end{multline*}
Arguing similarly to above, we deduce that $\tau_2^n= \beta h_n + O(h_n^2)$. Therefore the time $\tau^n$ to reach $\Bar{x}_{2}^n$ from $x_{0}$ satisfies
\[
\tau^n=(\rho+\beta) h_n+O(h_n^2) \leq 2\rho h_n + O(h_n^2). \qedhere
\]
\end{proof}

Let us now compare the trajectories $\Bar{\gamma}_{1}^n$ and $\gamma_{1}^n$ on $[t_1^n, t_2^n]$. Let $\delta_{1}^n(t)= \gamma_{1}^n(t)-\Bar{\gamma}_{1}^n(t)$. Hence, from the ODEs satisfied by the trajectories $\Bar{\gamma}_{1}^n$ and $\gamma_{1}^n$, we have 
\begin{equation*}
    \begin{aligned}
    \delta_{1}^n(t)&=\int_{t_1^n}^{t} \Big[k(s,\gamma_{1}^n(s))u_{1}^n(s)-k(s,\Bar{\gamma}_{1}^n(s))\Bar{u}_{1}^n \Big]\diff s\\
    &=\int_{t_1^n}^{t} \Big[k(s,\gamma_{1}^n(s))-k(s,\Bar{\gamma}_{1}^n(s))\Big]u_{1}^n(s) \diff s
    +\int_{t_1^n}^{t}k(s,\Bar{\gamma}_{1}^n(s))(u_{1}^n(s)-\Bar{u}_{1}^n) \diff s.
    \end{aligned}
\end{equation*}
Since $u_{1}^n$ is the optimal control, by Proposition~\ref{prop smooth 1}, it is Lipschitz continuous. Therefore, denoting by $L > 0$ the Lipschitz constant of $k$ on a bounded set containing the trajectories $\gamma_1^n$ and $\bar\gamma_1^n$ for every $n$, we have
$$
\abs*{\delta_{1}^n(t)} \le L \int_{t_1^n}^{t} \abs*{ \delta_{1}^n(s)} \diff s + K_{\max} \int_{t_1^n}^{t} L \abs*{s-t_1^n} \diff s,
$$
and hence, by using Grönwall's inequality,
$$
\abs*{\delta_{1}^n(t)} \le L K_{\max} \frac{(t-t_1^n)^2}{2} e^{L(t-t_1^n)}.
$$
In particular, if we set $t=t_1^n+h_n$, then
$$
\abs*{x_{2}^n- \Bar{x}_{2}^n} \le L K_{\max} \frac{h_n^2}{2} e^{L h_n} = O(h_n^2).
$$

Let $u_{3}^n = \frac{ x_{2}^n- \Bar{x}_{2}^n}{\abs*{ x_{2}^n- \Bar{x}_{2}^n}}$ (with the convention $x_3^n = 0$ if $x_2^n = \bar x_2^n$) and $\gamma_{3}^n$ be the solution of
\begin{equation}
    \left \{
    \begin{aligned}
    \Dot{\gamma}_{3}^n(t)&=k(t,\gamma_{3}^n(t)) u_{3}^n
    \\  
    \gamma_{3}^n(t_0 + \tau^n)&=\Bar{x}_{2}^n.       
    \end{aligned} \right.
\end{equation}
Using the lower bound $K_{\min}$ on $k$ and the fact that $\abs*{x_{2}^n- \Bar{x}_{2}^n} = O(h_n^2)$, one can easily deduce that the time $\sigma^n$ from $\Bar{x}_{2}^n$ to $ x_{2}^n$ along $\gamma_3^n$ (i.e., $\gamma_3^n(t_0 + \tau^n + \sigma^n) = x_2^n$) satisfies $\sigma^n = O(h_n^2)$.

We have thus constructed two ways to go from $x_0$ to $x_2^n$. The first one is to choose the path containing $x_{0}$, $x_{1}^n$, and $x_{2}^n$, which corresponds to the concatenation of the trajectories $\gamma_0$ on $[t_0, t_1^n]$ and $\gamma_1^n$ on $[t_1^n, t_2^n]$, and the second one is the path containing $x_{0}$, $\Bar{x}_{2}^n$, and $x_{2}^n$, which corresponds to the concatenation of the trajectories $\gamma_2^n$ on $[t_0, t_0 + \tau^n]$ and $\gamma_3^n$ on $[t_0 + \tau^n, t_0 + \tau^n + \sigma^n]$. Letting $T_1^n$ and $T_2^n$ be the times for going from $x_0$ to $x_2^n$ along these two paths, respectively, we have, by construction and the claim, that $T_1^n = 2 h_n$ and $T_2^n = \tau^n + \sigma^n \leq 2\rho h_n + O(h_n^2)$. Hence, since $\rho < 1$, we have, for $n$ large enough, that $T_2^n < T_1^n$.

From \eqref{eq:u0-in-W}, we deduce that
$$
\varphi(t_0, x_{0}) = \varphi(t_1^n, x_{1}^n) + h_n + o(h_n) = \varphi(t_2^n, x_{2}^n) + T_1^n + o(h_n),
$$
where the last equality comes from Proposition~\ref{dynamic prog-principle} and the fact that $\gamma_{1}^n \in \Opt(\Gamma, k, t_1^n, x_{1}^n)$. On the other hand, since the path from $x_0$ to $x_2^n$ going through $\bar x_2^n$ is an admissible trajectory for $k$, we have, by Proposition~\ref{dynamic prog-principle}, that
$\varphi(t_0, x_{0}) \le T_2^n + \varphi(t_0 + T_2^n, x_{2}^n)$. Hence
\begin{equation}\label{time contr}
    \varphi(t_2^n, x_{2}^n) + T_1^n + o(h_n) \le T_2^n + \varphi(t_0 + T_2^n, x_{2}^n).
\end{equation}
We also know that $t_0 + T_2^n < t_0 + T_1^n = t_2^n$ for $n$ large enough. Therefore, by Proposition~\ref{PropLowerBoundPartialT}, there exists a constant $c > 0$ such that
\begin{equation*}
    \varphi(t_2^n,x_{2}^n) > \varphi(t_0 +T_2^n,x_{2}^n) + (c-1) (t_2^n - t_0 - T_2^n) = \varphi(t_0+T_2^n, x_{2}^n) + (c-1) (T_1^n - T_2^n),
\end{equation*}
and, using \eqref{time contr}, we get $(c-1)(T_1^n-T_2^n)+T_1^n+o(h_n)\le T_2^n$, which leads to 
$$
2 h_n + o(h_n) = T_1^n+o(h_n) \le T_2^n \leq 2\rho h_n + O(h_n^2).
$$
Divide above inequality by $h_n$ to observe that
$$
2+o(1) \leq 2\rho+O(h_n).
$$
Finally by letting $n \to +\infty$, we conclude that $\rho\ge 1$, which is a contradiction. Therefore Case~2 will never happen and this ends the proof.
\end{proof}

Motivated by Proposition~\ref{PropWNormalizedGrad}, we introduce the following definition.

\begin{definition}
\label{DefNormalizedGradient}
Consider the optimal control problem $\OCP(\Gamma, k)$ and its value function $\varphi$ under the assumptions \ref{HypoOCP-Gamma}, \ref{HypoOCP-k-Bound}, and \ref{HypoOCP-k-loc-Lip} and let $\mathcal W$ be as in Definition~\ref{DefUW}. If $(t_0, x_0) \in \mathbbm R_+ \times \mathbbm R^d$ is such that $\mathcal{W}(t_0, x_0)$ contains exactly one element $-\omega_0$, then $\omega_0$ is called the \emph{normalized gradient} of $\varphi$ at $(t_0, x_0)$ and denoted by $\omega_0 = \widehat{\nabla\varphi}(t_0, x_0)$.
\end{definition}

As an immediate consequence of Proposition~\ref{PropUSingleElement} and  Theorem~\ref{thm Ut_0,x_0}, we obtain the following characterization of optimal controls.

\begin{corollary}\label{coro normalized}
Consider the optimal control problem $\OCP(\Gamma, k)$ and its value function $\varphi$ under the assumptions \ref{HypoOCP-Gamma}, \ref{HypoOCP-k-Bound}, and \ref{HypoOCP-k-loc-Lip}. Let $(t_0, x_0) \in \mathbbm R_+ \times \mathbbm R^d$, $\gamma \in \Opt(\Gamma, k, t_0, x_0)$, and $u$ be the optimal control associated with $\gamma$. Then, for every $t \in (t_0, t_{0}+\varphi(t_{0}, x_0))$, $\varphi$ admits a normalized gradient at $(t, \gamma(t))$ and $u(t) = -\widehat{\nabla \varphi}(t, \gamma(t))$, i.e.,
\begin{equation}
\label{eq:optimal-flow}
\Dot{\gamma}(t)= -k(t, \gamma(t)) \widehat{\nabla \varphi}(t, \gamma(t)).
\end{equation}

\end{corollary}

Combining Proposition~\ref{prop smooth 1} and Corollary~\ref{coro normalized}, for every optimal trajectory $\gamma$, we obtain that $t \mapsto \widehat{\nabla \varphi}(t,\gamma(t))$ is Lipschitz continuous for $t$ between the initial and exit times of $\gamma$. However, this provides no information on the regularity of $(t, x) \mapsto \widehat{\nabla \varphi}(t, x)$, which is the topic of our next result.

\begin{proposition}
\label{PropNormalizedGradientContinuous}
Consider the optimal control problem $\OCP(\Gamma, k)$ and its value function $\varphi$ under the assumptions \ref{HypoOCP-Gamma}, \ref{HypoOCP-k-Bound}, and \ref{HypoOCP-k-loc-Lip}. Then $\widehat{\nabla\varphi}$ is continuous on its domain of definition.
\end{proposition}

\begin{proof}
Let $\mathcal U$ be as in Definition~\ref{DefUW} and $D \subset \mathbbm R_+ \times (\mathbbm R^d \setminus \Gamma)$ be the domain of definition of $\widehat{\nabla\varphi}$, i.e., $D = \{(t_0, x_0) \in \mathbbm R_+ \times (\mathbbm R^d \setminus \Gamma) \suchthat \mathcal U(t_0, x_0) \text{ is a singleton}\}$. Let $(t_n, x_n)_{n \in \mathbbm N}$ be a sequence in $D$ converging as $n \to +\infty$ to some $(t_0, x_0) \in D$ and let $\bar u_n = -\widehat{\nabla\varphi}(t_n, x_n)$. We want to show that $\bar u_n \to -\widehat{\nabla\varphi}(t_0, x_0)$ as $n \to +\infty$ and, since $(\bar u_n)_{n \in \mathbbm N}$ is a sequence in the compact set $\mathbbm S^{d-1}$, it suffices to show that $-\widehat{\nabla\varphi}(t_0, x_0)$ is the unique adherent point of $(\bar u_n)_{n \in \mathbbm N}$. Let $\bar u_0$ be an adherent point of $(\bar u_n)_{n \in \mathbbm N}$ and consider a subsequence of $(\bar u_n)_{n \in \mathbbm N}$ converging to $\bar u_0$, which we still denote by $(\bar u_n)_{n \in \mathbbm N}$ for simplicity.

Since $\bar{u}_{n} \in \mathcal{U}(t_n, x_n)$, there exists a sequence of optimal trajectories $(\gamma_{n})_{n \in \mathbbm N}$, $\gamma_n \in \Opt(\Gamma, k, t_n, x_n)$, and a corresponding sequence of optimal controls $(u_{n})_{n \in \mathbbm N}$ such that $u_{n}(t_n) = \bar{u}_{n}$. From Proposition~\ref{prop smooth 1} and Arzelà--Ascoli Theorem, there exist elements $\gamma^*$ and $u^*$ such that, up to extracting a subsequence, $\gamma_{n} \to \gamma^*$ and $u_{n}\to u^*$ uniformly on compact time intervals. One immediately verifies that $u^\ast(t_0) = \bar u_0$, $\gamma^\ast(t_0) = x_0$, and that $\gamma^\ast \in \Opt(\Gamma, k, t_0, x_0)$ and $u^\ast$ is its associated optimal control, which shows that $\bar u_0 \in \mathcal U(t_0, x_0) = \{-\widehat{\nabla\varphi}(t_0, x_0)\}$, as required.
\end{proof}

\section{Minimal-time mean field games}\label{sec MFGs}

After having collected in Section~\ref{sec OCP} several preliminary results on the optimal control problem $\OCP(\Gamma, k)$, we now turn to the study of the main problem considered in the paper, the multi-population minimal-time mean field game $\MFG(\mathbf\Gamma, \mathbf K, \mathbf{m_0})$.
We address existence of equilibria in Section~\ref{SecMFGExist}, study their asymptotic behavior for large time in Section~\ref{SecMFGAsymptotic}, and characterize equilibria as solutions of a system of PDEs in Section~\ref{SecMFGSyst}.

Recall that, according to the presentation provided in Section~\ref{sec MFG model}, equilibria of $\MFG(\mathbf\Gamma,\allowbreak \mathbf K,\allowbreak \mathbf{m_0})$ are described in terms of vectors of measures $\mathbf Q = (Q_1, \dotsc, Q_N) \in \mathcal P(\mathbf C(\mathbbm R_+; \mathbbm R^d))^N$. Given such a vector of measures, we shall consider the $N$ optimal control problems $\OCP(\Gamma_i, k_{\mathbf Q, i})$, with $k_{\mathbf Q, i}$ given by $k_{\mathbf Q, i}(t, x) = K_i(m_t^i, \hat m_t^i, x)$ for $(t, x) \in \mathbbm R_+ \times \mathbbm R^d$ and where $m_t^i = {e_t}_{\#} Q_i$ and $\hat m_t^i$ is defined in \eqref{eq:hat_m_i}. We will denote the value function of $\OCP(\Gamma_i, k_{\mathbf Q, i})$ by $\varphi_{\mathbf Q, i}$, and we omit $\mathbf Q$ from the notation of both $k_{\mathbf Q, i}$ and $\varphi_{\mathbf Q, i}$ when it is clear from the context. For simplicity of notation, we also write $\Adm_i(\mathbf Q)$ for $\Adm(k_i)$ and $\Opt_i(\mathbf\Gamma, \mathbf Q, t_0, x_0)$ for $\Opt(\Gamma_i, k_i, t_0, x_0)$.

\subsection{Existence of equilibria}
\label{SecMFGExist}

The goal of this part is to establish existence of equilibria for $\MFG(\mathbf\Gamma, \mathbf K, \mathbf{m_0})$, which is done by recasting the existence of an equilibrium in terms of the existence of a fixed point of a certain set-valued map and applying a suitable fixed-point theorem. This section follows closely \cite[Section~5]{Mazanti2019Minimal} but, due to the facts that assumptions \ref{HypoMFG-Gamma}--\ref{HypoMFG-K-Lip} are weaker than those from \cite{Mazanti2019Minimal} and that we work here with mean field games in the non-compact state space $\mathbbm R^d$, several proofs must be adapted in a nontrivial way to the present setting. The main result to be proved in this section is the following.

\begin{theorem}\label{exis-Equilib}
Consider the mean field game $\MFG(\mathbf\Gamma, \mathbf K, \mathbf{m_0})$ under assumptions \ref{HypoMFG-Gamma}--\ref{HypoMFG-K-Lip}. Then there exists an equilibrium $\mathbf Q \in \mathcal{P}(\mathbf{C}(\mathbbm R_+; \mathbbm R^d))^N$ for $\MFG(\mathbf\Gamma, \mathbf K, \mathbf{m_0})$.
\end{theorem}

Let us start by showing an additional continuity property of the value function.

\begin{lemma}
\label{LemmValueFunctionContinuous}
Consider the mean field game $\MFG(\mathbf\Gamma, \mathbf K, \mathbf{m_0})$ under the assumptions \ref{HypoMFG-Gamma}--\ref{HypoMFG-K-Lip}. Then, for every $i \in \{1, \dotsc, N\}$, $(t, x, \mathbf Q) \mapsto \varphi_{\mathbf Q, i}(t, x)$ is continuous on $\mathbbm R_+ \times \mathbbm R^d \times \mathcal P(\mathbf C(\mathbbm R_+; \mathbbm R^d))^N$.
\end{lemma}

\begin{proof}
Fix $i \in \{1, \dotsc, N\}$ and let $(t_n, x_n, \mathbf Q_n)_{n \in \mathbbm N}$ be a sequence taking values in $\mathbbm R_+ \times \mathbbm R^d \times \mathcal P(\mathbf C(\mathbbm R_+; \mathbbm R^d))^N$ converging to some $(t_\ast, x_\ast, \mathbf Q_\ast)$, and denote $\mathbf Q_n = (Q_{1, n}, \dotsc, Q_{N, n})$ and $\mathbf Q_\ast = (Q_{1, \ast}, \dotsc,\allowbreak Q_{N, \ast})$. For $n \in \mathbbm N$ and $(t, x) \in \mathbbm R_+ \times \mathbbm R^d$, define $k_n(t, x) = K(m_{n, t}^i, \hat m_{n, t}^i, x)$ and $k_\ast(t, x) = K(m_{\ast, t}^i, \hat m_{\ast, t}^i, x)$, where $m_{n, t}^i = {e_t}_{\#} Q_{n, i}$, $m_{\ast, t}^i = {e_t}_{\#} Q_{\ast, i}$, and $\hat m_{n, t}^i$ and $\hat m_{\ast, t}^i$ are defined as in \eqref{eq:hat_m_i}. Note that, by continuity of $Q \mapsto {e_t}_{\#} Q$, we have that $k_n(t, x) \to k_\ast(t, x)$ for every $(t, x) \in \mathbbm R_+ \times \mathbbm R^d$. For simplicity of notation, we write $\varphi_n$ and $\varphi_\ast$ for $\varphi_{\mathbf Q_n, i}$ and $\varphi_{\mathbf Q_{\ast}, i}$, respectively.

By Proposition~\ref{psi}, $(\varphi_n(t_n, x_n))_{n \in \mathbbm N}$ is a bounded sequence and thus, to prove that it converges to $\varphi_\ast(t_\ast, x_\ast)$, it suffices to show that $\varphi_\ast(t_\ast, x_\ast)$ is the unique adherent point of $(\varphi_n(t_n, x_n))_{n \in \mathbbm N}$. Let $\kappa_\ast$ be an adherent point of $(\varphi_n(t_n, x_n))_{n \in \mathbbm N}$ and consider the subsequence of $(\varphi_n(t_n, x_n))_{n \in \mathbbm N}$ which converges to $\kappa_\ast$, which we still denote by $(\varphi_n(t_n, x_n))_{n \in \mathbbm N}$ for simplicity.

For $n \in \mathbbm N$, let $\gamma_n \in \Opt_i(\mathbf\Gamma, \mathbf Q_n, t_n, x_n)$. Since $(\gamma_n)_{n \in \mathbbm N}$ is an equibounded and equi-Lipschitz sequence, by Arzelà--Ascoli Theorem, up to extracting a subsequence, which we still denote by $(\gamma_n)_{n \in \mathbbm N}$, there exists $\gamma_\ast \in \Lip_{K_{\max}}(\mathbbm R_+; \mathbbm R^d)$ such that $\gamma_n \to \gamma_\ast$ as $n \to +\infty$ (uniformly on compact time intervals). For every $t_1, t_2 \in \mathbbm R_+$ with $t_1 < t_2$, we have $\abs*{\frac{\gamma_n(t_2) - \gamma_n(t_1)}{t_2 - t_1}} \leq \frac{1}{t_2 - t_1} \int_{t_1}^{t_2} k_n(s, \gamma_n(s)) \diff s$ and, using \ref{HypoMFG-K-Lip} and letting $n \to +\infty$, we deduce that $\abs*{\frac{\gamma_\ast(t_2) - \gamma_\ast(t_1)}{t_2 - t_1}} \leq \frac{1}{t_2 - t_1} \int_{t_1}^{t_2} k_\ast(s, \gamma_\ast(s)) \diff s$, yielding that $\gamma_\ast \in \Adm_i(\mathbf Q_\ast)$. Moreover, since $\gamma_n(t_n) = x_n$, $\gamma_n$ is constant on $[0, t_n]$ and $[t_n + \varphi_n(t_n, x_n), +\infty)$, and $\gamma_n(t_n + \varphi_n(t_n, x_n)) \in \Gamma$ for every $n \in \mathbbm N$, we easily deduce that $\gamma_\ast(t_\ast) = x_\ast$, $\gamma_\ast$ is constant on $[0, t_\ast]$ and $[t_\ast + \kappa_\ast, +\infty)$, and $\gamma_\ast(t_\ast + \kappa_\ast) \in \Gamma$, yielding in particular that $\varphi_\ast(t_\ast, x_\ast) \leq \kappa_\ast$.

Let us assume, to obtain a contradiction, that $\varphi_\ast(t_\ast, x_\ast) < \kappa_\ast$. For simplicity, let $\varsigma_n = t_n + \varphi_\ast(t_\ast, x_\ast)$, $\varsigma_\ast = t_\ast + \varphi_\ast(t_\ast, x_\ast)$, $\xi_n = \gamma_n(\varsigma_n)$, and $\xi_\ast = \gamma_\ast(\varsigma_\ast) \in \Gamma$. Note that, since $\varphi_\ast(t_\ast, x_\ast) < \kappa_\ast$, we have $\xi_n \notin \Gamma$ for $n$ large enough. Let $\tilde\gamma_n \in \mathbf C(\mathbbm R_+; \mathbbm R^d)$ be defined by
\[
\tilde\gamma_n(t) = 
\begin{dcases}
\gamma_n(t) & \text{if } 0 \leq t \leq \varsigma_n, \\
\xi_n + \frac{\xi_\ast - \xi_n}{\abs{\xi_\ast - \xi_n}} K_{\min} (t - \varsigma_n) & \text{if } \varsigma_n \leq t \leq \varsigma_n + \frac{\abs{\xi_\ast - \xi_n}}{K_{\min}}, \\
\xi_\ast & \text{if } t \geq \varsigma_n + \frac{\abs{\xi_\ast - \xi_n}}{K_{\min}}.
\end{dcases}
\]
Clearly, $\tilde\gamma_n \in \Adm_i(\mathbf Q)$ and $\tau_\Gamma(t_n, \tilde\gamma_n) \leq \varphi_\ast(t_\ast, x_\ast) + \frac{\abs{\xi_\ast - \xi_n}}{K_{\min}}$. Since $\varphi_\ast(t_\ast, x_\ast) < \kappa_\ast$, we have $\tau_\Gamma(t_n, \tilde\gamma_n) < \frac{\varphi_\ast(t_\ast, x_\ast) + \kappa_\ast}{2} < \kappa_\ast$ for $n$ large enough, implying that $\varphi_n(t_n, x_n) < \frac{\varphi_\ast(t_\ast, x_\ast) + \kappa_\ast}{2} < \kappa_\ast$ for $n$ large enough and contradicting thus the fact that $\varphi_n(t_n, x_n) \to \kappa_\ast$ as $n \to +\infty$. Hence, one has necessarily $\varphi_n(t_n, x_n) \to \varphi_\ast(t_\ast, x_\ast)$ as $n \to +\infty$, as required.
\end{proof}

The next result, which is an immediate consequence of Proposition~\ref{psi}, states an a priori property of equilibria.

\begin{lemma}\label{phi(R)}
Consider the mean field game $\MFG(\mathbf\Gamma, \mathbf K, \mathbf{m_0})$, assume that \ref{HypoMFG-Gamma} and \ref{HypoMFG-K-Bound} are satisfied, and let $\phi$ be the function from Notation~\ref{Notation-Phi}. Then there exists a nondecreasing function $\psi: \mathbbm R_+ \to \mathbbm R_+$ such that, for every equilibrium $\mathbf Q = (Q_1, \dotsc, Q_N) \in \mathcal P(\mathbf C(\mathbbm R_+; \mathbbm R^d))^N$ of $\MFG(\mathbf\Gamma, \mathbf K, \mathbf{m_0})$, $t \geq 0$, $i \in \{1, \dotsc, N\}$, and $R > 0$, we have
\[
Q_i\left(\Lip_{K_{\max}}(\mathbbm R_+; B_{\psi(R)})\right) \geq \phi(R).
\]
In particular, denoting $m_t^i = {e_t}_{\#} Q_i$, we have
$m_t^{i}(B_{\psi(R)}) \ge \phi(R)$.
\end{lemma}

Lemma~\ref{phi(R)} shows that it suffices to look for equilibria of $\MFG(\mathbf\Gamma, \mathbf K, \mathbf{m_0})$ in the set
\begin{multline}
\label{eq:defi-Q}
\mathfrak{Q}= \left\{\mathbf Q = (Q_1, \dotsc, Q_N) \in \mathcal{P}(\mathbf{C}(\mathbbm R_+; \mathbbm R^d))^N \suchthat {e_0}_{\#} \mathbf Q = \mathbf{m_0} \text{ and } \right. \\ \left.\forall i \in \{1, \dotsc, N\},\, \forall R >0, \, Q_i\left(\Lip_{K_{\max}}(\mathbbm R_+; B_{\psi(R)})\right) \ge \phi(R)\right\},
\end{multline}
where $\phi$ and $\psi$ are as in the statement of Lemma~\ref{phi(R)}. We next provide elementary properties of $\mathfrak Q$.

\begin{lemma}\label{LemmQNonemptyConvexCompact}
Consider the mean field game $\MFG(\mathbf\Gamma, \mathbf K, \mathbf{m_0})$, assume that \ref{HypoMFG-Gamma} and \ref{HypoMFG-K-Bound} are satisfied, and let $\mathfrak Q$ be the set defined in \eqref{eq:defi-Q}. Then $\mathfrak Q$ is nonempty, convex, and compact with respect to the topology of weak convergence of measures.
\end{lemma}

\begin{proof}
The set $\mathfrak Q$ is clearly convex and, to see that it is nonempty, define $b: \mathbbm R^d \to \mathbf{C}(\mathbbm R_+; \mathbbm R^d)$ as the function which associates with each $x \in \mathbbm R^d$ the function $b(x)$ given by $b(x)(t) = x$ for every $t\in \mathbbm R_+$. It is immediate to check that ${b}_{\#} \mathbf{m_0} \in \mathfrak{Q}$, and hence $\mathfrak{Q}$ is nonempty.

To prove that $\mathfrak Q$ is compact, notice first that $\mathfrak Q = \mathfrak Q_1^N \cap \mathfrak Q_2$, where
\[\mathfrak Q_1 = \left\{Q \in \mathcal P(\mathbf C(\mathbbm R_+; \mathbbm R^d)) \suchthat \forall R > 0,\, Q\left(\Lip_{K_{\max}}(\mathbbm R_+; B_{\psi(R)})\right) \geq \phi(R)\right\}\]
and $\mathfrak Q_2 = \{\mathbf Q \in \mathcal P(\mathbf C(\mathbbm R_+; \mathbbm R^d))^N \suchthat {e_0}_{\#}\mathbf Q = \mathbf{m_0}\}$. Since $\mathbf Q \mapsto {e_0}_{\#} \mathbf Q$ is continuous, $\mathfrak Q_2$ is closed, and hence it suffices to show that $\mathfrak Q_1$ is compact. By Prokhorov Theorem (see, e.g., \cite[Theorem~5.1.3]{Ambrosio1}), it suffices to show that $\mathfrak Q_1$ is tight and closed. Tightness of $\mathfrak Q_1$ follows immediately from the facts that $\phi(R) \to 1$ as $R \to +\infty$ and that, by Arzelà--Ascoli Theorem, for every $R > 0$, $\Lip_{K_{\max}}(\mathbbm R_+; B_{\psi(R)})$ is compact in the topology of uniform convergence on compact sets.

To see that $\mathfrak Q_1$ is closed, let $(Q_{n})_{n \in \mathbbm N}$ be a sequence in $\mathfrak{Q}_1$ converging to some $Q \in \mathcal{P}(\mathbf{C}(\mathbbm R_+; \mathbbm R^d))$. For every $R > 0$, $\Lip_{K_{\max}}(\mathbbm R_+; B_{\psi(R)})$ is closed and thus, by using \cite[Theorem~2.1]{Billingsley},
one obtains
$$
Q(\Lip_{K_{\max}}(\mathbbm R_+; B_{\psi(R)})) \ge \limsup_{n \to \infty} Q_{n}(\Lip_{K_{\max}}(\mathbbm R_+; B_{\psi(R)})) \ge \phi(R),
$$
which proves that $Q \in \mathfrak{Q}_1$. Hence $\mathfrak{Q}_1$ is closed.
\end{proof}

We now recast the definition of equilibrium of $\MFG(\mathbf\Gamma, \mathbf K, \mathbf{m_0})$ in terms of fixed points of a set-valued map defined on $\mathfrak Q$. Let $F:\mathfrak{Q} \rightrightarrows \mathfrak{Q}$ associate with each $\mathbf Q \in \mathfrak{Q}$ the subset $F(\mathbf Q)$ of $\mathfrak Q$ defined by
\begin{multline}
\label{eq:defi-F}
F(\mathbf Q)= \left\{ \Tilde{\mathbf Q} = (\Tilde Q_1, \dotsc, \Tilde Q_N) \in \mathfrak{Q} \suchthat \right.\\
\left.\forall i \in \{1, \dotsc, N\},\, \Tilde{Q}_i\text{-almost every $\gamma$ satisfies } \gamma \in \Opt_i(\mathbf\Gamma, \mathbf Q, 0, \gamma(0)) \right\}.
\end{multline}
Clearly, $\mathbf Q \in \mathfrak Q$ is an equilibrium of $\MFG(\mathbf\Gamma, \mathbf K, \mathbf{m_0})$ if and only if it is a fixed point of $F$, i.e., $\mathbf Q \in F(\mathbf Q)$. For every $i \in \{1, \dotsc, N\}$, we consider the set $\Opt_i(\mathbf Q) \subset \mathbf C(\mathbbm R_+; \mathbbm R^d)$ containing all optimal trajectories of the $i$-th population for $\mathbf Q$ and starting at time $0$, i.e.,
\begin{equation}\label{OPT}
\Opt_i(\mathbf Q) = \bigcup_{x_0 \in \mathbbm R^d} \Opt_i(\mathbf\Gamma, \mathbf Q, 0, x_0).
\end{equation}
The set $F(\mathbf Q)$ can be rewritten in terms of $\Opt_i(\mathbf Q)$ as
\begin{equation}\label{Fm0}
F(\mathbf Q) = \left\{\Tilde{\mathbf Q} = (\Tilde Q_1, \dotsc, \Tilde Q_N) \in \mathfrak{Q} \suchthat \forall i \in \{1, \dotsc, N\},\, \Tilde{Q}_i(\Opt_i(\mathbf Q)) = 1\right\}.
\end{equation}

\begin{lemma}\label{Opt uppersemi}
Consider the mean field game $\MFG(\mathbf\Gamma, \mathbf K, \mathbf{m_0})$ under the assumptions \ref{HypoMFG-Gamma}--\ref{HypoMFG-K-Lip} and let $\mathfrak Q$ and $\Opt_i$, $i \in \{1, \dotsc, N\}$, be defined as in \eqref{eq:defi-Q} and \eqref{OPT}. For every $R>0$ and $i \in \{1, \dotsc, N\}$, define $\widehat\Opt_{i, R}: \mathfrak{Q} \rightrightarrows \mathbf C(\mathbbm R_+; \mathbbm R^d)$ by 
$$ 
\widehat{\Opt}_{i,R}(\mathbf Q) = \Opt_i(\mathbf Q) \cap \Lip_{K_{\max}}(\mathbbm R_+; B_{\psi(R)}).
$$
Then $\widehat{\Opt}_{i,R}$ is upper semicontinuous.
\end{lemma}

The proof of Lemma~\ref{Opt uppersemi} is based on the continuity of the value function from Lemma~\ref{LemmValueFunctionContinuous} and follows the same lines as that of \cite[Lemma~5.4]{Mazanti2019Minimal}, being thus omitted here.

\begin{lemma}\label{comp Fm0}
Consider the mean field game $\MFG(\mathbf\Gamma, \mathbf K, \mathbf{m_0})$ under the assumptions \ref{HypoMFG-Gamma}--\ref{HypoMFG-K-Lip} and let $\mathfrak Q$ and $F$ be defined as in \eqref{eq:defi-Q} and \eqref{eq:defi-F}, respectively. Then $F$ is upper semicontinuous and, for every $\mathbf Q \in \mathfrak{Q}$, $F(\mathbf Q)$ is nonempty, convex, and compact.
\end{lemma}

\begin{proof}
Given $\mathbf Q \in \mathfrak Q$, it follows immediately from \eqref{Fm0} that $F(\mathbf Q)$ is convex, and one can easily prove that it is nonempty and compact by adapting the arguments of \cite[Lemma~5.3]{Mazanti2019Minimal} (see also \cite[Lemma~4.7(a)]{Dweik2020Sharp}).

Since $\mathfrak Q$ is compact and $F$ has closed values, to prove that $F$ is upper semicontinuous it is sufficient to show that its graph is closed. Let $(\mathbf Q_n)_{n \in \mathbbm N}$ be a sequence in $\mathfrak{Q}$ with $\mathbf Q_n \to \mathbf Q$ for some $\mathbf Q \in \mathfrak{Q}$ and $(\Tilde{\mathbf Q}_n)_{n \in \mathbbm N}$ be a sequence in $\mathfrak{Q}$ with $\Tilde{\mathbf Q}_n \in F(\mathbf Q_n)$ for every $n\in \mathbbm N$ and $\Tilde{\mathbf Q}_n \to \Tilde{\mathbf Q}$ for some $\Tilde{\mathbf Q} \in \mathfrak{Q}$. We denote $\mathbf Q_n = (Q_{n, 1}, \dotsc, Q_{n, N})$ and $\tilde{\mathbf Q}_n = (\tilde Q_{n, 1}, \dotsc, \tilde Q_{n, N})$.

For each $n \in \mathbbm N$, since $\Tilde{\mathbf Q}_n \in F(\mathbf Q_n)$, we have $\tilde{Q}_{n, i}(\Opt_i(\mathbf Q_n)) = 1$ for every $i \in \{1, \dotsc, N\}$ and, since $\tilde{\mathbf Q}_n \in \mathfrak Q$, we also have that $\tilde{Q}_{n, i}(\Lip_{K_{\max}}(\mathbbm R_+; B_{\psi(R)})) \geq \phi(R)$ for every $i \in \{1, \dotsc, N\}$ and $R > 0$, where $\phi$ is the function from Notation~\ref{Notation-Phi}. Hence $\tilde{Q}_{n, i}(\widehat\Opt_{i, R}(\mathbf Q_n)) \geq \phi(R)$ for every $n \in \mathbbm N$, $i \in \{1, \dotsc, N\}$, and $R > 0$.

For every $\epsilon \in (0, 1)$ and $i \in \{1, \dotsc, N\}$, let  
$V_{\epsilon}^{i} = \{\gamma \in \mathbf{C}(\mathbbm R_+; \mathbbm R^d) \suchthat \mathbf d(\gamma, \Opt_i(\mathbf Q)) \le \epsilon\}$,
where $\mathbf d$ is given by \eqref{eq:defi-d-cont}, and note that $V_\epsilon^i$ is a neighborhood of $\widehat\Opt_{i, R}(\mathbf Q)$ for every $R > 0$. Let $R_0 > 0$ be such that $\phi(R_0) \geq 1 - \epsilon$.
Since $\widehat{\Opt}_{i,R_0}$ is upper semicontinuous by Lemma~\ref{Opt uppersemi}, there exists a neighborhood $W_{\epsilon}$ of $\mathbf Q$ in $\mathfrak{Q}$ such that $\widehat{\Opt}_{i,R_0}(\hat{\mathbf Q}) \subset V_{\epsilon}^{i}$ for every $i \in \{1, \dotsc, N\}$ and $\hat{\mathbf Q} \in W_{\epsilon}$. From the convergence $\mathbf Q_n \to \mathbf Q$, one concludes that there exists $N_{\epsilon}$ such that, for every $n \ge N_{\epsilon}$, one has $\mathbf Q_n \in W_{\epsilon}$, and thus $\widehat{\Opt}_{i,R_0}(\mathbf Q_n) \subset V_{\epsilon}^{i}$.

Since $\Tilde{Q}_{n, i}(\widehat{\Opt}_{i,R_0}(\mathbf Q_n)) \geq \phi(R_0) \geq 1 - \epsilon$, one obtains that $\Tilde{Q}_{n, i}(V_{\epsilon}^{i}) \geq 1 - \epsilon$ for every $n \ge N_{\epsilon}$ and $i \in \{1, \dotsc, N\}$. Since $\Tilde{\mathbf Q}_n \to \Tilde{\mathbf Q}$ and $V_{\epsilon}^{i}$ is closed, we have $\Tilde{Q}_{i}(V_{\epsilon}^{i}) \ge \limsup_{n\to \infty} \Tilde{Q}_{n, i}(V_{\epsilon}^{i}) \geq 1 - \epsilon$. On the other hand, since $\Opt_i(\mathbf Q)$ is closed and $(V_\epsilon^i)_{\epsilon \in (0, 1)}$ is a nondecreasing family of sets with $\bigcap_{\epsilon \in (0, 1)} V_{\epsilon}^i = \Opt_i(\mathbf Q)$, we conclude that $\tilde Q_i(\Opt_i(\mathbf Q)) = \lim_{\epsilon \to 0} \Tilde{Q}_{i}(V_{\epsilon}^{i}) = 1$. Hence $\Tilde{\mathbf Q} \in F(\mathbf Q)$, which concludes the proof that the graph of $F$ is closed.
\end{proof}

Let us now conclude the proof of Theorem~\ref{exis-Equilib}.

\begin{proof}[Proof of Theorem~\ref{exis-Equilib}]
By Lemmas~\ref{LemmQNonemptyConvexCompact} and \ref{comp Fm0} and Kakutani fixed point theorem (see, e.g., \cite[\S~7, Theorem~8.6]{GranasDugandji}), $F$ admits a fixed point, i.e., there exists $\mathbf Q \in \mathfrak{Q}$ such that $\mathbf Q \in F(\mathbf Q)$, which means $\mathbf Q$ is an equilibrium for $\MFG(\mathbf\Gamma, \mathbf K, \mathbf{m_0})$.
\end{proof}

\begin{remark}
Theorem~\ref{exis-Equilib} asserts the existence of an equilibrium for $\MFG(\mathbf\Gamma, \mathbf K, \mathbf{m_0})$, but uniqueness does not necessarily hold. An example of this fact in the single-population case is presented in \cite[Remark~7.1]{Mazanti2019Minimal} under the assumption $K_1 \equiv 1$, in which there is no interaction between agents. 

Let us provide a heuristic example illustrating why uniqueness is not expected in the multi-population case even when agents interact. Consider the case $d = N = 2$, $m_{0}^{1}$ is the uniform measure on $B((-1,0), R)$, $m_{0}^{2}$ is the uniform measure on $B((1,0),R)$, $0< R < 1$, $\Gamma_1= B((1,0),R)$, and $\Gamma_2= B((-1,0), R)$, and assume that $K_1$ and $K_2$ are such that agents are more penalized by the other population than by their own population, i.e., $K_i(\mu, \nu, x) < K_i(\nu, \mu, x)$ if $\nu$ is larger than $\mu$ in a neighbourhood of $x$, for $i =1,\, 2$. In this case, we may expect heuristically the phenomenon of \emph{lane formation}, in which the populations will group in separate lanes, so that each population gets to its target set while avoiding interaction with the other population (see, for instance, \cite{Cristiani2014Multiscale, Gibelli2018Crowd} for more details on lane formation in other kinds of models for crowd motion and in experiments). If the lanes at an equilibrium are asymmetric (which is expected if our model reproduces the behaviour usually observed in experiments), then we obtain another different equilibrium with the same initial conditions by performing the symmetry transformation $(x_1, x_2) \mapsto (x_1, -x_2)$, and hence we do not expect uniqueness of equilibrium in this case.      
\end{remark}

\subsection{Asymptotic behavior}
\label{SecMFGAsymptotic}

In this part, we characterize the behavior of $m_{t}^i$ as $t\to +\infty$, where $m_t^i = {e_t}_{\#} Q_i$ and $\mathbf Q = (Q_1, \dotsc, Q_N)$ is an equilibrium of $\MFG(\mathbf\Gamma, \mathbf K, \mathbf{m_0})$. Intuitively, one expects $m_t^i$ to converge to a measure concentrated on the target set $\Gamma_i$ and, in addition to proving this result in the general case, we also provide convergence rates when the initial measure $m_0^i$ has finite $p$ moments for some $p \in [1, +\infty)$ and prove finite-time convergence when the initial measure has bounded support.

In order to characterize the limit of $m_t^i$ as $t \to +\infty$, let us introduce some notation. Let $\mathbf C_{\lim}(\mathbbm R_+; \mathbbm R^d) = \{\gamma \in \mathbf C(\mathbbm R_+; \mathbbm R^d) \suchthat \lim_{t \to +\infty} \gamma(t) \text{ exists and is finite}\}$, which is a Borel subset of $\mathbf C(\mathbbm R_+; \mathbbm R^d)$, and define $e_{\infty}: \mathbf C_{\lim}(\mathbbm R_+; \mathbbm R^d) \to \mathbbm R^d$ by $e_\infty(\gamma) = \lim_{t \to +\infty} \gamma(t)$, which is a Borel-measurable function. By definition of optimal trajectories, $\Opt_i(\mathbf Q) \subset \mathbf C_{\lim}(\mathbbm R_+; \mathbbm R^d)$ for every $\mathbf Q \in \mathcal P(\mathbf C(\mathbbm R_+; \mathbbm R^d))^N$, and thus ${e_\infty}_{\#} \mathbf Q \in \mathcal P(\mathbbm R^d)^N$ is well-defined for every equilibrium $\mathbf Q$ of a mean field game $\MFG(\mathbf\Gamma, \mathbf K, \mathbf{m_0})$.

We are now in position to state and prove the main result of this section.

\begin{theorem}
\label{ThmAsymp}
Consider the mean field game $\MFG(\mathbf\Gamma, \mathbf K, \mathbf{m_0})$ under assumptions \ref{HypoMFG-Gamma} and \ref{HypoMFG-K-Bound}. Let $\mathbf Q = (Q_1, \dotsc, Q_N) \in \mathcal P(\mathbf C(\mathbbm R_+; \mathbbm R^d))^N$ be an equilibrium of $\MFG(\mathbf\Gamma,\allowbreak \mathbf K,\allowbreak \mathbf{m_0})$, $\mathbf{m}_t = (m_t^1, \dotsc, m_t^N)$ be defined by $\mathbf m_t = {e_t}_{\#} \mathbf Q$ for $t \in [0, +\infty]$, and $\psi$ be the function whose existence is asserted in Proposition~\ref{psi}.

\begin{enumerate}
\item\label{ThmAsymp-Converg} For every $i \in \{1, \dotsc, N\}$, we have $m_t^i \to m_\infty^i$ as $t \to +\infty$.
\item\label{ThmAsymp-Wp} Let $p \in [1, +\infty)$, $i\in \{1, \dotsc, N\}$, and assume that $m_0^{i} \in \mathcal P_p(\mathbbm R^d)$. Then, for every $t \in [0, +\infty]$, we have $m_t^i \in \mathcal P_p(\mathbbm R^d)$. Moreover, there exist constants $\alpha > 0$ and $t_0 \geq 0$ such that
\begin{equation}
\label{eq:cv-rate-Wp}
\mathbf W_p(m_t^i, m_\infty^i)^p \leq 2^p \int_{\mathbbm R^d \setminus B_{\alpha(t - t_0)}} \psi(\abs{x})^p \diff m_0^i(x), \qquad \forall t \geq t_0.
\end{equation}

\item\label{ThmAsymp-FiniteTime} Let $i \in \{1, \dotsc, N\}$ and assume that $m_0^i$ is compactly supported. Then, for every $t \in [0, +\infty]$, $m_t^i$ is compactly supported and there exists $\tau \geq 0$ such that
\[
m_t^i = m_\infty^i, \qquad \forall t \geq \tau.
\]
\end{enumerate}
\end{theorem}

\begin{remark}
Note that, by Proposition~\ref{psi}, $\psi$ has linear growth and thus, together with the assumption that $m_0^i \in \mathcal P_p(\mathbbm R^d)^N$, one immediately obtains that the right-hand side of \eqref{eq:cv-rate-Wp} tends to $0$ as $t \to +\infty$. When more information on the distribution of $m_0^i$ is available, the right-hand side of \eqref{eq:cv-rate-Wp} allows one to obtain estimates on the convergence rate of $m_t^i$ as $t \to +\infty$ in the Wasserstein distance.
\end{remark}

\begin{proof}[Proof of Theorem~\ref{ThmAsymp}]
To show \ref{ThmAsymp-Converg}, let $f: \mathbbm R^d \to \mathbbm R$ be continuous and bounded and fix $i \in \{1, \dotsc, N\}$. We then have, using the continuity and boundedness of $f$ and Lebesgue's dominated convergence theorem, that
\begin{multline*}
\int_{\mathbbm R^d} f(x) \diff m_t^i(x) = \int_{\mathbf{C}_{\lim}(\mathbbm R_+; \mathbbm R^d)} f(\gamma(t)) \diff Q_i(\gamma) \\ \xrightarrow[t \to +\infty]{} \int_{\mathbf C_{\lim}(\mathbbm R_+; \mathbbm R^d)} f\Bigl(\lim_{t \to +\infty} \gamma(t)\Bigr) \diff Q_i(\gamma) = \int_{\mathbbm R^d} f(x) \diff m_\infty^i(x),
\end{multline*}
yielding the required convergence.

Let us now prove \ref{ThmAsymp-Wp}. For $t \in [0, +\infty]$, we have, using Proposition~\ref{psi}, that
\begin{multline*}
\int_{\mathbbm R^d} \abs{x}^p \diff m_t^i(x) = \int_{\Opt_i(\mathbf Q)} \abs{\gamma(t)}^p \diff Q_i(\gamma) \\ \leq \int_{\Opt_i(\mathbf Q)} \psi(\abs{\gamma(0)})^p \diff Q_i(\gamma) = \int_{\mathbbm R^d} \psi(\abs{x})^p \diff m_0^i(x),
\end{multline*}
where $\gamma(\infty)$ is defined as $\lim_{t \to +\infty} \gamma(t)$. Since $\psi$ has linear growth, it follows that $ m_t^i \in \mathcal P_p(\mathbbm R^d)$ for every $t \in [0, +\infty]$.

Let $T$ be the function whose existence is asserted in Proposition~\ref{psi} and $\alpha > 0$, $t_0 \geq 0$ be such that $T(R) \leq \frac{R}{\alpha} + t_0$ for every $R > 0$. Let $t \in [t_0, +\infty)$. Note that, using the notations introduced in Section~\ref{SecNotation}, we have $(e_t, e_\infty)_{\#} Q_i \in \Pi(m_t^i, m_\infty^i)$ and thus, by \eqref{eq:defi-Wasserstein}, we have
\begin{equation*}
    \mathbf{W}_p(m_t^i, m_\infty^i)^p \le
		\int_{\mathbbm R^d \times \mathbbm R^d} \abs{x - y}^p \diff\, (e_t, e_\infty)_{\#} Q_i(x, y) =
		\int_{\Opt_i(\mathbf Q)} \abs*{e_t(\gamma)-e_{\infty}(\gamma)}^p \diff Q_i(\gamma).
\end{equation*}
If $\gamma \in \Opt_i(\mathbf Q)$ is such that $\abs{\gamma(0)} \leq \alpha(t - t_0)$, then, since $T(\abs{\gamma(0)}) \leq t$, we have, as a consequence of Proposition~\ref{psi}, that $\gamma(t) \in \Gamma_i$ and $\gamma$ is constant on $[t, +\infty)$, yielding that $e_t(\gamma) = e_\infty(\gamma)$. Thus
\begin{equation*}
    \mathbf{W}_p(m_t^i, m_\infty^i)^p \leq \int_{\mathbf{Opt}_i(\mathbf Q) \cap \left\{\gamma \suchthat \gamma(0) \notin B_{\alpha(t - t_0)}\right\}} \abs*{e_t(\gamma)-e_{\infty}(\gamma)}^p \diff Q_i(\gamma).
\end{equation*}
By using the fact from Proposition~\ref{psi} that $e_t(\gamma) \in B_{\psi(\abs*{\gamma(0)})}$ for all $t\in [0, +\infty]$ and $\gamma \in \Opt_i(\mathbf Q)$, one has $\abs*{e_t(\gamma)} \le \psi(\abs*{\gamma(0)})$ and thus
    \begin{align*}
    \mathbf{W}_p(m_t^i, m_\infty^i)^p & \le \int_{\mathbf{Opt}_i(Q)\cap \left\{\gamma \suchthat \gamma(0) \notin B_{\alpha(t - t_0)}\right\}} 2^p \psi(\abs*{\gamma(0)})^p \diff Q_i(\gamma) \displaybreak[0]\\
		& = 2^p \int_{\mathbbm R^d \setminus B_{\alpha(t - t_0)}} \psi(\abs{x})^p \diff m_0^i(x),
    \end{align*}
as required.

Finally, to prove \ref{ThmAsymp-FiniteTime}, let $R_0 > 0$ be such that the support of $m_0^i$ is included in $B_R$ and notice that, as a consequence of Proposition~\ref{psi}, the support of $m_t^i$ is included in $B_{\psi(R_0)}$ for every $t \in [0, +\infty]$. Letting $T$ be as in the statement of Proposition~\ref{psi} and $\tau = T(R_0)$, we deduce that, for every $t \geq \tau$ and $\gamma \in \Opt_i(\mathbf Q)$ with $\abs{\gamma(0)}\leq R_0$, we have $e_t(\gamma) = e_\infty(\gamma)$, which concludes the proof since $Q_i$ is supported in $\Opt_i(\mathbf Q) \cap \{\gamma \suchthat \abs{\gamma(0)} \leq R_0\}$.
\end{proof}

\subsection{The MFG system}
\label{SecMFGSyst}

As a final step in the study of $\MFG(\mathbf\Gamma, \mathbf K, \mathbf{m_0})$, we characterize its equilibria as solutions of a system of partial differential equations, called the \emph{MFG system}. Given an equilibrium $\mathbf Q = (Q_1, \dotsc, Q_N)$, by Proposition~\ref{thm H-J}, the value functions $\varphi_{\mathbf Q, i}$, $i \in \{1, \dotsc, N\}$, corresponding to each population are already known to satisfy a Hamilton--Jacobi equation, and we are thus left to prove that the measures $m_t^i = {e_t}_{\#} Q_i$ are also solutions of suitable partial differential equations. Since $Q_i$ is concentrated on optimal trajectories, which satisfy \eqref{eq:optimal-flow} thanks to Corollary~\ref{coro normalized}, one expects $t \mapsto m_t^i$ to be a solution to a continuity equation with velocity field $-\widehat{\nabla\varphi}_{\mathbf Q, i}$.

In order for the above reasoning to be made precise, one must verify that the assumptions of Corollary~\ref{coro normalized} are satisfied. Since \ref{HypoOCP-k-loc-Lip} requires $k$ to be locally Lipschitz continuous both in time and space, we shall make here the following stronger assumption on $\MFG(\mathbf\Gamma, \mathbf K, \mathbf{m_0})$.

\begin{hypotheses}[resume]
\item\label{HypoMFG-K-loc-Lip-Wass} There exists $p \geq 1$ such that $\mathbf{m_0} \in \mathcal P_p(\mathbbm R^d)^N$ and, for every $i \in \{1, \dotsc, N\}$, $K_i: \mathcal P_p(\mathbbm R^d) \times \mathcal P_p(\mathbbm R^d) \times \mathbbm R^d \to \mathbbm R_+$ is Lipschitz continuous with respect to all its variables (using the Wasserstein distance $\mathbf W_p$ in $\mathcal P_p(\mathbbm R^d)$) and locally in the last variable, i.e., for every $R > 0$, there exists $L > 0$ such that, for every $(\mu_1, \nu_1, x_1), (\mu_2, \nu_2, x_2) \in \mathcal P_p(\mathbbm R^d) \times \mathcal P_p(\mathbbm R^d)^{N-1} \times B_R$, we have
\[
\abs{K_i(\mu_1, \nu_1, x_1) - K_i(\mu_2, \nu_2, x_2)} \leq L \left(\mathbf W_p(\mu_1, \mu_2) + \mathbf W_p(\nu_1, \nu_2) + \abs{x_1 - x_2}\right).
\]
\end{hypotheses}

\begin{remark}
If $\mathbf Q = (Q_1, \dotsc, Q_N) \in \mathcal P(\mathbf C(\mathbbm R_+; \mathbbm R^d))^N$ is such that $Q_i(\Lip_{c}(\mathbbm R_+; \mathbbm R^d)) = 1$ for some $c > 0$ and every $i \in \{1, \dotsc, N\}$, and if ${e_t}_{\#} Q_i \in \mathcal P_p(\mathbbm R^d)$ for some $p \geq 1$ and every $t \geq 0$ and $i \in \{1, \dotsc, N\}$, then one immediately verifies, by considering the coupling measure $(e_t, e_s)_{\#} Q_i \in \Pi(m_t^i, m_s^i)$ in \eqref{eq:defi-Wasserstein}, that $t \mapsto {e_t}_{\#} Q_i$ is Lipschitz continuous with respect to the distance $\mathbf W_p$ in $\mathcal P_p(\mathbbm R^d)$. Hence, if $\MFG(\mathbf\Gamma, \mathbf K, \mathbf{m_0})$ satisfies \ref{HypoMFG-K-loc-Lip-Wass} and $\mathbf Q$ is an equilibrium of $\MFG(\mathbf\Gamma, \mathbf K, \mathbf{m_0})$, the corresponding optimal control problems $\OCP(\Gamma_i, k_{\mathbf Q, i})$, $i \in \{1, \dotsc, N\}$, satisfy \ref{HypoOCP-k-loc-Lip}.
\end{remark}

\begin{theorem}
\label{TheoMFGSyst}
Consider the mean field game $\MFG(\mathbf\Gamma, \mathbf K, \mathbf{m_0})$ under assumptions \ref{HypoMFG-Gamma}, \ref{HypoMFG-K-Bound}, and \ref{HypoMFG-K-loc-Lip-Wass} and assume that $\mathbf Q = (Q_1, \dotsc, Q_N) \in \mathcal P(\mathbf C(\mathbbm R_+; \mathbbm R^d))^N$ is an equilibrium of $\MFG(\mathbf \Gamma, \mathbf K, \mathbf{m_0})$. Consider the value functions $\varphi_i = \varphi_{\mathbf Q,i}$ and the time-dependent measures $m_i(t, \cdot) = m_t^i = {e_t}_{\#} Q_i$ for $i \in \{1, \dotsc, N\}$. Then $(m^1, \dotsc, m^N, \varphi_{1}, \dotsc, \varphi_{N})$ solves the MFG system
\begin{equation}\label{MFGs system}
\begin{aligned}
    \left \{
    \begin{aligned}
    & \partial_t m_i(t, x) - \diverg\Big (m_i(t, x) K_i(m_t^i, \hat{m}_t^i, x) \widehat{\nabla \varphi}_{i}(t, x)\Big) = 0, & & (t, x)\in \mathbbm R_+^* \times (\mathbbm R^d \setminus \Gamma_i),  \\
    & -\partial_t \varphi_{i} (t, x) + \abs*{\nabla \varphi_{i}(t, x)} K_i(m_t^i, \hat{m}_t^i, x) - 1 = 0, & & (t, x)\in \mathbbm R_+ \times (\mathbbm R^d \setminus \Gamma_i),  \\
    & m_i(0, \cdot) = m_0^i, \\
    & \varphi_{i}(t, x) = 0, & & (t, x) \in \mathbbm R_+\times \Gamma_i,
    \end{aligned} \right. \\
\end{aligned}
\end{equation}
for all $i \in \{1, \dotsc, N\}$, where the first and second equations are satisfied, respectively, in the sense of distributions and in the viscosity sense.
\end{theorem}

Note that the Hamilton--Jacobi equations on $\varphi_i$ and the corresponding boundary conditions follow immediately from Proposition~\ref{thm H-J}, and the continuity equations on $m_i$ can be established using \eqref{eq:optimal-flow} and the fact that, from Proposition~\ref{PropUSingleElement}, Theorem~\ref{thm Ut_0,x_0}, Definition~\ref{DefNormalizedGradient}, and Proposition~\ref{PropNormalizedGradientContinuous}, $\widehat{\nabla\varphi}_i$ is continuous on the support of $m_t^i$. We refer to \cite[Theorem~6.1]{Mazanti2019Minimal} and \cite[Theorem~4.12]{Dweik2020Sharp} for more details on the proof in the case of a single population, but we stress the fact, contrarily to those references, we establish here \eqref{eq:optimal-flow}, and hence the continuity equations on $m_i$, under weaker assumptions on $K_i$ and without relying on semiconcavity properties of $\varphi_i$. Notice also that the coupling between the different populations occur through the terms $\hat m_t^i$, which are defined in \eqref{eq:hat_m_i}.

Theorem~\ref{TheoMFGSyst} shows that any equilibrium $\mathbf Q$ of a mean field game $\MFG(\mathbf\Gamma, \mathbf K, \mathbf{m_0})$ satisfies the MFG system \eqref{MFGs system}. To prove that \eqref{MFGs system} actually characterizes equilibria of $\MFG(\mathbf\Gamma, \mathbf K, \mathbf{m_0})$, we also need a converse statement, namely that solutions of \eqref{MFGs system} yield equilibria of $\MFG(\mathbf\Gamma, \mathbf K, \mathbf{m_0})$. Such a converse statement has been sketched in \cite[Remark~6.1]{Mazanti2019Minimal} for single-population minimal-time mean field games. We now provide a more detailed argument in our present setting.

\begin{theorem}
Consider the mean field game $\MFG(\mathbf\Gamma, \mathbf K, \mathbf{m_0})$ under assumptions \ref{HypoMFG-Gamma}, \ref{HypoMFG-K-Bound}, and \ref{HypoMFG-K-loc-Lip-Wass}, and assume in addition that $\mathbf{m_0} = (m_0^1, \dotsc, m_0^N)$ is such that $m_0^i$ is compactly supported for every $i \in \{1, \dotsc, N\}$. For $i \in \{1, \dotsc, N\}$, let $\varphi_i: \mathbbm R_+ \times \mathbbm R^d \to \mathbbm R_+$ and $t \mapsto m_i(t, \cdot) \in \mathcal P(\mathbbm R^d)$ be continuous functions.
% \begin{enumerate}
%     \item Let $i\in \{1, \dotsc, N\}$ and assume that $\varphi_i$ satisfies the second and fourth equations of \eqref{MFGs system} in the viscosity sense. Then $\varphi_i$ is the value function of $\OCP(\Gamma_i,k_i)$, where $k_i: \mathbbm R_+ \times \mathbbm R^d \to \mathbbm R_+$ is defined for $(t, x) \in \mathbbm R_+ \times \mathbbm R^d$ by $k_i(t, x) = K_i(m_t^i, \hat m_t^i, x)$, and $\hat m_t^i$ is defined as in \eqref{eq:hat_m_i}. In particular, $\widehat{\nabla \varphi}_i$ exists and is continuous at $(t,\gamma(t))$ for every optimal trajectory $\gamma$ of $\OCP(\Gamma_i, k_i)$ starting at some time $t_0 \in \mathbbm R_+$ and every $t\in (t_0, t_0+ \varphi_i(t_0,\gamma(t_0)))$. 
%     \item Assume in addition that $\widehat{\nabla \varphi}_i$ is
%     $m_i$ satisfies the first and third equation of \eqref{MFGs system} in the sense of distributions
% \end{enumerate}
Assume that, for every $i\in \{1, \dotsc, N \}$ and $t>0$, $\widehat{\nabla \varphi}_i(t, \cdot)$ exists and is continuous in the support of $m_i(t, \cdot)$, and that $(m_1, \dotsc ,m_N, \varphi_1 ,\dotsc ,\varphi_N)$ satisfies \eqref{MFGs system}, where the first equation is satisfied in the sense of distributions and the second equation is satisfied in the viscosity sense. Then there exists an equilibrium $\mathbf Q = (Q_1, \dotsc, Q_N) \in \mathcal P(\mathbf C(\mathbbm R_+; \mathbbm R^d))^N$ of $\MFG(\mathbf\Gamma, \mathbf K, \mathbf{m_0})$ such that, for every $i \in \{1, \dotsc, N\}$, $m_t^i = m_i(t, \cdot) = {e_t}_{\#} Q_i$ for every $t \geq 0$ and $\varphi_i$ is the value function of $\OCP(\Gamma_i, k_{\mathbf Q, i})$.
\end{theorem}

\begin{proof}
Let $k_i: \mathbbm R_+ \times \mathbbm R^d \to \mathbbm R_+$ be defined for $(t, x) \in \mathbbm R_+ \times \mathbbm R^d$ by $k_i(t, x) = K_i(m_t^i, \hat m_t^i, x)$, where $\hat m_t^i$ is defined as in \eqref{eq:hat_m_i}, and consider the optimal control problem $\OCP(\Gamma_i, k_i)$. Since $\varphi_i$ is lower bounded by $0$, satisfies the second equation of \eqref{MFGs system} in the viscosity sense, and also satisfies the fourth equation of \eqref{MFGs system}, we deduce from \cite[Chapter~IV, Corollary~4.3]{BardiDolcetta} that $\varphi_i$ is the value function of $\OCP(\Gamma_i, k_i)$. Note that \cite[Chapter~IV, Corollary~4.3]{BardiDolcetta} is stated for autonomous control systems, but it can be applied to the non-autonomous control system $\dot\gamma(t) = k_i(t, \gamma(t)) u(t)$ by considering the augmented state $\tilde x(t) = (t, \gamma(t))$. Moreover, \cite[Chapter~IV, Corollary~4.3]{BardiDolcetta} assumes that the target set has compact boundary, but we get the conclusion in our framework by reasoning locally and using Proposition~\ref{psi}.

For $i \in \{1, \dotsc, N\}$, since $m_i$ satisfies the continuity equation in \eqref{MFGs system} in the sense of distributions and the corresponding velocity field is bounded by $K_{\max}$, it follows from the Superposition Principle for continuity equations (see \cite[Theorem~3.2]{Ambrosio2008Transport}) that there exists $Q_i \in \mathcal P(\mathbf C(\mathbbm R_+; \mathbbm R^d))$ such that $m_t^i = {e_t}_{\#} Q_i$ for every $t \geq 0$. Let $\mathbf Q = (Q_1, \dotsc, Q_N)$. Note that, since $m_0^i$ is compactly supported and the velocity field in the continuity equation is bounded, $m_t^i$ is also compactly supported, and thus $\mathbf Q \in \mathfrak Q$. We will show that $\mathbf Q$ is an equilibrium of $\MFG(\mathbf\Gamma, \mathbf K, \mathbf{m_0})$ by showing that $\mathbf Q \in F(\mathbf Q)$, i.e., that $Q_i$ is supported on $\Opt_i(\mathbf Q)$.

To see that, notice that, from the proof of \cite[Theorem~3.2]{Ambrosio2008Transport}, it also follows that $Q_i$ is concentrated on the solutions of $\Dot{\gamma}(t) = -k_i(t, \gamma(t)) \widehat{\nabla \varphi}_{i}(t, \gamma(t))$, which are clearly admissible trajectories for $\OCP(\Gamma_i, k_i)$ since $\abs{\widehat{\nabla \varphi}_{i}(t, \gamma(t))} = 1$.
We prove that such trajectories are optimal by showing that they satisfy the equality in the dynamic programming principle \eqref{eq:DPP}. Let $\gamma$ be such a trajectory and notice that it is $K_{\max}$ Lipschitz continuous. From the definition of normalized gradient, we have
$$
\lim_{h\to 0^+} \frac{\varphi_{i}\bigl(t+h, \gamma(t) - h k_i(t, \gamma(t))\widehat{\nabla \varphi}_{i} (t,\gamma(t))\bigr) -\varphi_{i}(t,\gamma(t))}{h}=-1.
$$
Using the facts that 
$\gamma(t+h)=\gamma(t) + h \Dot{\gamma}(t)+o(h)$
and that $\varphi_i$ is Lipschitz continuous, we deduce that
$$
\lim_{h\to 0^+} \frac{\varphi_{i}(t+h, \gamma(t+h))-\varphi_{i}(t,\gamma(t))}{h}=-1.
$$
Since $t\mapsto \varphi_{i}(t, \gamma(t))$ is Lipschitz continuous, and hence differentiable almost everywhere, we deduce that
$\frac{\diff}{\diff t} \varphi_{i}(t,\gamma(t))=-1$
a.e., and thus, integrating the above expression from $t$ to $t + h$, we get that 
$\varphi_{i}(t+h, \gamma(t+h)) - \varphi_{i}(t, \gamma(t)) = -h$,
and therefore, by Proposition~\ref{dynamic prog-principle}, $\gamma$ is optimal for $\OCP(\Gamma_i, k_i)$. Hence $Q_i$ is concentrated on $\Opt_i(\mathbf Q)$, concluding the proof that $\mathbf Q$ is an equilibrium.
\end{proof}

\bibliographystyle{abbrv}
\bibliography{main}

\end{document}